\documentclass[11pt]{amsart}
\usepackage{amsthm}
\theoremstyle{plain}
\newtheorem{theorem}{Theorem}
\newtheorem*{theorem*}{Theorem}

\usepackage{geometry}           
\usepackage{graphicx}
\usepackage{caption}
\usepackage{amssymb}
\usepackage{amsthm}
\usepackage{epstopdf}
\usepackage{mathrsfs}  
\usepackage{color}
\usepackage{float}
\usepackage{hyperref}
\usepackage{doi}
\usepackage{xcolor}
\usepackage{subcaption}
\usepackage{makecell}

\usepackage{amsaddr}

\DeclareMathOperator{\arccosh}{arccosh}
\DeclareMathOperator{\arcsinh}{arcsinh}

\title[Dirichlet domains and NEC uniformization of extremal hyperbolic surfaces]{Dirichlet domains and the NEC uniformization of extremal hyperbolic surfaces}

\author{Ernesto Girondo \and Claudia Mu\~noz}

\address{Departamento de Matem\'aticas, Universidad Aut\'onoma de Madrid and Instituto de Ciencias Matem\'aticas ICMAT (CSIC-UAM-UCM-UC3M). 28049-Madrid, Spain. \ \texttt{ernesto.girondo@uam.es, claudia.munnoz@uam.es}}

\usepackage{mdframed}

\newtheorem{corollary}{Corollary}
\newtheorem{lemma}{Lemma}
\newtheorem{proposition}{Proposition}
\newtheorem{definition}{Definition}

\theoremstyle{remark}
\newtheorem{remark}{Remark}
\newtheorem{example}{Example}

\begin{document}

\renewcommand{\thefootnote}{\fnsymbol{footnote}} 
\renewcommand{\thefootnote}{\arabic{footnote}}

\begin{abstract}
Extremal hyperbolic surfaces are finite-type hyperbolic surfaces that contain an embedded disc of the largest possible radius. We introduce a new general approach to these objects, based on the uniformization by NEC groups, obtaining a structural description of the Dirichlet domains associated to extremal disc centers for surfaces with cusps and/or geodesic boundary. We also provide a number of applications of our new description of extremal surfaces, as an effective counting formula for extremal disc configurations, the determination of every such configuration with Euler characteristic $-1$ and $-2$, the location of hidden extremal discs in several families of extremal surfaces, and the full description of  their automorphism groups.
\end{abstract}

\maketitle

\section{Introduction and statement of results}

A fundamental geometric invariant of hyperbolic surfaces is their maximal injectivity radius, namely the radius of their largest embedded metric disc. Hyperbolic surfaces that attain the maximum possible value of this invariant for a fixed topology are called \emph{extremal surfaces}. 

Since the work of Bavard \cite{Bavard}, closed extremal surfaces have played an important role in the interaction between hyperbolic geometry, discrete group actions and extremal metric problems. For closed hyperbolic surfaces, the theory of Fuchsian and NEC groups proved to be a particularly effective tool in the study of extremal discs, allowing the determination of extremal surfaces, the computation of their automorphism groups and the location of their extremal discs (see \cite{Girondo_Gonzalez-Diez_1999}, \cite{Girondo_Gonzalez-Diez_2002}, \cite{Girondo_Gonzalez-Diez_2002_2}, \cite{Girondo_Nakamura_2007}, \cite{Nakamura_2002}, \cite{Nakamura_2005}, \cite{Nakamura_2009}, \cite{Nakamura_2012}, \cite{Nakamura_2013}, \cite{Nakamura_2016}).

More recently, J.~deBlois introduced a fundamentally different approach to the study of extremal radii based on the \emph{centered dual complex}. This point of view led to sharp bounds for the maximal injectivity radius of hyperbolic surfaces of finite type with cusps (\cite{deblois2015}, \cite{deblois2018}) and, later, with geodesic boundary (\cite{deblois-romanelli}),  providing  a complete understanding of the extremal radius as a function that depends only on the topology of the surface. 

The main purpose of this article is to recover the theory of extremal surfaces within the classical setting of the uniformization by NEC groups. 
An informal description of our first main result is the following: 

\begin{quote}
Let $S$ be an extremal hyperbolic surface, maybe  with cusps and/or geodesic boundary, uniformized by a NEC group $\Gamma$, and let $z$ be a lift of an extremal disc center in $S$. Then the Dirichlet domain $D_{\Gamma}(z)$ presents certain rigid geometric features that are determined by the topology of the surface.
\end{quote}

The precise statement is given in \autoref{theo2}. 
Our starting observation is therefore that extremal discs determine remarkably rigid Dirichlet domains, whose geometry and combinatorics are subjected to strong constraints determined by the topology of the underlying surface.  In particular, extremal surfaces can be studied through the geometry of a single Dirichlet domain with a really specific shape, together with a finite side-pairing pattern. This description extends to surfaces with cusps and boundary the role traditionally played by regular Dirichlet polygons in the study of closed extremal surfaces. Our approach leads to a NEC group characterization of extremal surfaces that complements the centered-dual-complex viewpoint and provides an effective framework for addressing some classification and counting problems.

Our second principal contribution is what we call the  EBF-correspondence, which links extremal disc configurations to subgroups of suitable chosen extended triangle groups. This way we transform the question of counting extremal disc configurations into a group-theoretic problem, yielding to explicit formulas and an effective computational tool for the construction of extremal surfaces.

Finally, we show that the geometry of these Dirichlet domains can be used to locate the hidden extremal discs, determine all extremal surfaces within certain families and compute their automorphism groups.

The novelty of our approach is that it recovers the geometry of general extremal surfaces directly from the action of their uniformizing NEC groups, as in the case of closed surfaces. This provides a bridge between the centered-dual-complex description provided by DeBlois and Romanelli and the classical uniformization techniques that have traditionally been used in the study of closed extremal surfaces.

The structure of the paper is as follows. In \autoref{sec:extremalsurfaces} we introduce extremal surfaces and recall the geometric framework established by DeBlois and Romanelli. Section \ref{sec:uniformization} contains the main structural result of the paper, namely the description of the Dirichlet domains associated with extremal disc centers. In \autoref{sec:extremaldiscconfigurations} we introduce the EBF-correspondence and obtain counting results for extremal disc configurations, including a complete analysis of the cases with Euler characteristic $\chi=-1$ and $\chi=-2$. Section \ref{section:location} develops the theory of admissible displacement cycles and applies it to the location of additional extremal discs. Finally, in \autoref{sec:counting} we combine these techniques to determine explicitely several families of extremal surfaces and their automorphism groups.

\section{Extremal surfaces} 
\label{sec:extremalsurfaces}

It is a well-known fact that the injectivity radius of a hyperbolic surface of finite volume is bounded by a constant that depends solely on the topology of the surface. In  order to describe this bound explicitly, let us denote by $\alpha(r)$ the angle of an equilateral 
triangle with sidelength $2r$, by $\beta(r)$ the positive angle of a horocyclic ideal triangle with compact side of length $2r$ and by $\gamma(r)$ the angle of an  equilateral regular quadrilateral with sidelength $2r$. The following statement is just a rewriting of Theorem 1.1 in \cite{deblois-romanelli}, which provides an upper bound for the radius of a disc embedded in a general hyperbolic surface in terms of its topology and gives a characterization of the surfaces that contain a disc with that maximal radius:

\begin{theorem}[\cite{deblois-romanelli}, Theorem 1.1]\label{theo1} Given integer numbers $\chi<0$ and $n,b\geq 0$, the radius of a metric disc embedded in a complete hyperbolic surface $S$ with  Euler characteristic $\chi$, $b$ geodesic boundary components and $n$ cusps cannot be larger than  $r_{\chi,n,b}$, the unique positive solution to the equation 

\vspace{-0.7em}

\begin{equation}\label{deblois}
\left(6-6\chi-3n-3b\right)\alpha\big(r_{\chi,n,b}\big)+2n\beta\big(r_{\chi,n,b}\big)+2b\gamma\big(r_{\chi,n,b}\big)=2\pi.
\end{equation}

\vspace{0.3em}

Moreover, $S$ contains an embedded metric disc of radius $r_{\chi,n,b}$ centered at a point $p\in S$ if and only if $S$ admits a cell decomposition such that all the vertices lying in the interior of $S$ agree with $p$ and the 2-cells are either equilateral triangles of sidelength $2r_{\chi,n,b}$, horocyclic ideal triangles of compact sidelength $2r_{\chi,n,b}$ or Saccheri quadrilaterals with legs of length $r_{\chi,n,b}$ and summit of length $2r_{\chi,n,b}$. The horocyclic ideal triangles and the Saccheri quadrilaterals are, respectively, in one-to-one correspondence with the cusps of $S$ and the boundary components of $S$. Each boundary component is the base of one of these quadrilaterals, so it has length $2\arcsinh(\tanh(r_{\chi,n,b}))$.

 For each triplet $(\chi,n,b)$, there exists a hyperbolic surface with Euler characteristic $\chi$,  $n$ cusps and $b$ boundary components that admits an embedded disc of maximal radius $r_{\chi,n,b}$.
\end{theorem}

Recall that a Saccheri quadrilateral is a 4-gon that has two opposite sides (\emph{legs}) of equal length, perpendicular to a third edge (\emph{base}), and such as the only remaining edge (\emph{summit}) intersects both legs with the same angle, called the \emph{summit angle}. For example, dividing an equilateral regular quadrilateral of sidelength $2r$ along the common perpendicular to a pair of opposite sides produces a Saccheri quadrilateral with summit length $2r$, legs length $r$ and summit angle $\gamma(r)$. We will denote this particular Saccheri quadrilateral by $SQ^{\gamma(r)}_{2r}$.

\begin{remark}\label{angles}
One can easily find explicit expressions of the angles $\alpha$, $\beta$ and $\gamma$ in the previous theorem as functions of the variable $r$:

\vspace{-0.7em}

\begin{equation*}
    \alpha(r)=2\sin^{-1}\left(\frac{1}{2\cosh r}\right),\;\; \beta(r)=\sin^{-1}\left(\frac{1}{\cosh r}\right),\;\; \gamma(r)=2\sin^{-1}\left(\frac{1}{\sqrt{2}\cosh r}\right).
\end{equation*}

    As a consequence, the relations

\vspace{-0.7em}
\begin{equation*}
    \sin(\beta)=2\sin(\alpha/2),\; \;\sin(\gamma/2)=\sqrt{2}\sin(\alpha/2),\; \;\cos(\alpha)=\cos^2(\gamma/2)
\end{equation*}
hold and, in particular, we see that $\alpha(r)<\gamma(r)<2\beta(r)$ for every $r>0$. Additionally, one can easily show that $2\beta(r)<3\alpha(r)$, $2\gamma(r)<3\alpha(r)$ and $\alpha(r)<\beta(r)<\gamma(r)$.
\end{remark}

\begin{definition}
    A complete hyperbolic surface $S$ with Euler characteristic $\chi$, $n$ cusps and $b$ boundary components is called an \emph{extremal surface} if there exists an embedded \emph{(extremal)} disc in $S$ of maximum radius $r_{\chi,n,b}$.
\end{definition}

Given a triplet $(\chi, n, b)$, DeBlois and Romanelli describe in \cite{deblois-romanelli} one concrete extremal surface of Euler characteristic $\chi$, $n$ cusps and $b$ boundary components, by a particular choice of how to arrange a number of triangles and quadrilaterals. The way we are going to handle here any extremal surface relies on a manipulation of such decomposition of any extremal surface into pieces. It is illustrative to apply our ideas first to the surface constructed in \cite{deblois-romanelli} before giving in the next section a detailed account of how our construction works in the general case.

Let $\Omega$ be the polygon on the left of  \autoref{fig:fig_const1}, which has been constructed by arranging four equilateral triangles of sidelength $2r$ with a common vertex at the origin, and then gluing to one of them the summit of a Saccheri quadrilateral $SQ^{\gamma}_{2r}$, and to another of them the compact side of a horocyclic ideal triangle of angle $\beta(r)$, with $r=r_{\chi,n,b}$. A solid line inside $\Omega$ represents an  orientation-preserving isometry that identifies the connected edges of $\Omega$. These transformations plus the hyperbolic reflection fixing the base of the Saccheri quadrilateral generate the group $\Gamma$ that uniformizes the orientable extremal surface $S$ of genus one with one cusp and one boundary component introduced in \cite{deblois-romanelli}.

As the $\Gamma$-images of $\Omega$ tessellate the hyperbolic disc $\mathbb{D}$, there is a complete neighborhood $\Omega'$ of the origin given by the union of a minimum number (six, in this case) of copies of $\Omega$. Let us denote by $P'\subset \Omega'$ the subpolygon formed by the $\Gamma$-images of the initial triangles and quadrilaterals that meet at the origin.
 Note that  if we sum the total angle at the origin of this collection of triangles and quadrilaterals, by \eqref{deblois} we have

\vspace{-0.7em}
 
$$
12\alpha(r)+2\beta(r)+2\gamma(r)=\big(6-6\chi-3n-3b\big)\alpha(r)+2n\beta(r)+2b\gamma(r)=2\pi.
$$

\vspace{0.2em}

\begin{figure}[!htbp]
\begin{subfigure}{.24\textwidth}
  \centering
  \includegraphics[width=0.96\linewidth]{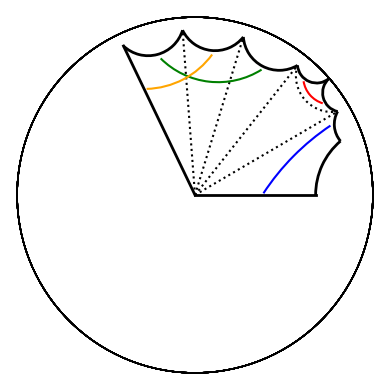}
\end{subfigure} 
\begin{subfigure}{.24\textwidth}
  \centering
  \includegraphics[width=0.96\linewidth]{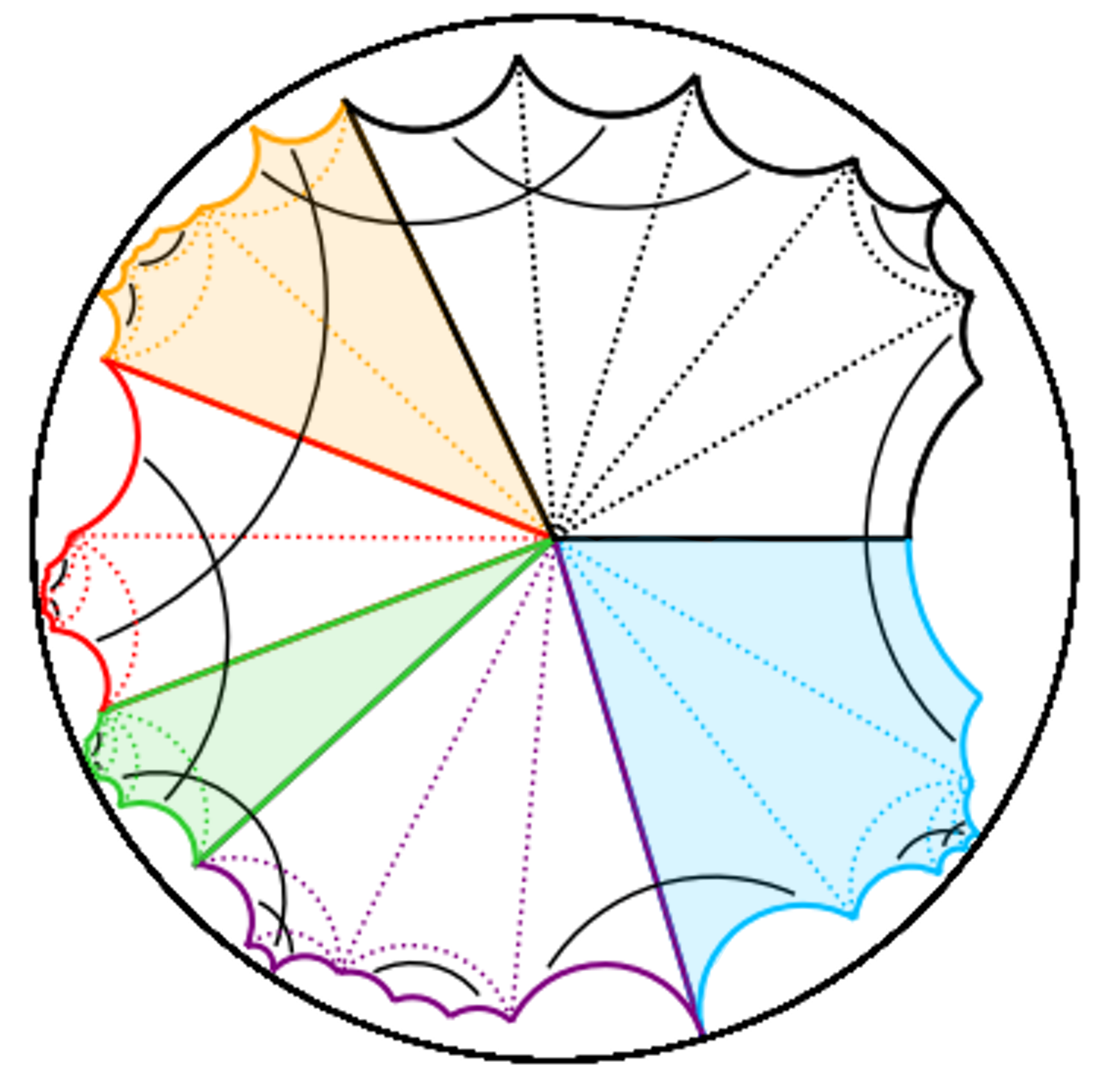}
\end{subfigure}
\begin{subfigure}{.24\textwidth}
  \centering
  \includegraphics[width=0.96\linewidth]{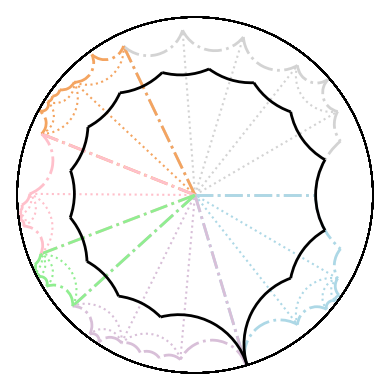}
\end{subfigure}%
\begin{subfigure}{.24\textwidth}
  \centering
  \includegraphics[width=0.96\linewidth]{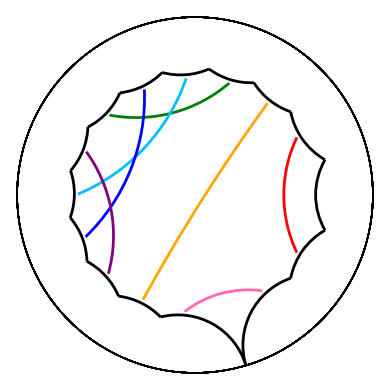}
\end{subfigure}
\caption{From left to right, the hyperbolic polygons $\Omega$ and $\Omega'$ defined in a particular case with  $\chi=-2$ and $n=b=1$, and the construction of a new fundamental domain determined by a collection of perpendicular bisectors.}
\label{fig:fig_const1}
\end{figure}

Let us now illustrate, for the particular case depicted in \autoref{fig:fig_const1}, the general construction we will describe in the next section.

Let $\mathscr{B}$ be the collection of the geodesic perpendicular bisectors of all the compact inner edges of the triangles that occur in the decomposition of $P'$ (all these edges have length $2r$), together with the bases of the Saccheri quadrilaterals in that decomposition. The cyclic pairwise intersection of the geodesic arcs in $\mathscr{B}$ determines the vertices of a new polygon $P$, illustrated in the third image  of \autoref{fig:fig_const1}.

Finally, a careful inspection of the action of the elements of $\Gamma$ on the edges of this new polygon $P$ produces some side-pairing identifications, represented in the right-hand side of \autoref{fig:fig_const1} by lines connecting the corresponding pairs of edges. These side-pairing identifications, together with the hyperbolic reflection in the remaining edge of $P$, generate a new NEC group $\Gamma'$, which is clearly a subgroup of $\Gamma$.

By construction, all the positive angles of  $P$ equal  $2\pi/3$ except those corresponding to the two vertices belonging to the bases of the two Saccheri quadrilaterals, that equal $\pi/2$.
In turn, the $\Gamma'$-orbits of vertices of $P$ have three representatives, with the obvious exception of the orbit of length one consisting simply of the unique improper vertex, and the orbit of length 2  given precisely by the two vertices in the bases of the Saccheri quadrilaterals.
 By Poincaré's Polygon Theorem, $P$ is a fundamental domain for $\Gamma'$.

Moreover, an area computation gives 

\vspace{-0.5em}

$$
\begin{array}{rl}
 \text{Area}(P)&=12(2\pi-\alpha-\pi-2\pi/3)+2(\pi-\pi/2-\beta)+2(2\pi-3\pi/2-\gamma)\\
 & =4\pi-12\alpha+\pi-2\beta+\pi-2\gamma=4\pi=-2\pi\chi=\text{Area}(S)
\end{array}
$$

\vspace{0.2em}

\noindent where we have used \eqref{deblois} and the Gauss-Bonnet Theorem, and we have denoted $\chi$ the Euler characteristic of $S=\mathbb{D}/\Gamma$. It follows that $\Gamma=\Gamma'$, so  the polygon $P$ just constructed is also a fundamental domain for $\Gamma$, which will be more suitable for our approach to extremal surfaces than the original fundamental domain $\Omega$. We generalize this construction to any arbitrary extremal surface in the next section.

\section{The uniformization of extremal surfaces}
\label{sec:uniformization}

 The following result provides a complete structural description of the Dirichlet domains for NEC groups corresponding to extremal disc centers, valid for surfaces with geodesic boundary and/or cusps. This yields a group‑theoretic characterization of extremal surfaces that complements the centered‑dual-complex viewpoint, producing an optimal framework for addressing some fundamental questions later on:
 
\begin{theorem}\label{theo2}
Let $S\simeq \mathbb{D}/\Gamma$ be an extremal hyperbolic surface of Euler characteristic $\chi<0$ with $n$ cusps and $b$ boundary components, uniformized by a NEC group $\Gamma$. If $z\in\mathbb{D}$ is the preimage of an extremal disc center in $S$ by the quotient map  $\pi:\mathbb{D}\longrightarrow \mathbb{D}/\Gamma \simeq S$, the Dirichlet domain $D_\Gamma(z)$ is a hyperbolic $m$-gon for 
$m=6-6\chi-2n-b$ with four different types of  edges distributed as:

\begin{itemize}
\item $n$ pairs of non-compact (type I) edges, one for each cusp. Each pair consists of two consecutive edges of $D_\Gamma(z)$ that are  identified by a parabolic element of $\Gamma$. 
    \item $b$ (type II) edges of length 
$
L=2\arcsinh\big(\tanh(r_{\chi,n,b})\big),
$ 
which account for the boundary components of $S$. Each 
of them is contained in the fixed point set of a hyperbolic reflection belonging to $\Gamma$.
\item $2b$ (type III) edges  of length 
$
l^*=\arccosh\big(\sqrt{2\cos(\alpha)}\big)+\arccosh\big(2\cos(\alpha/2)/\sqrt{3}\big),
$ 
for $\alpha=\alpha(r_{\chi,n,b})$, 
pairwise identified by  hyperbolic elements of $\Gamma$. Each pair consists of the two direct neighbors of one edge of type II.
\item $6-6\chi-4n-4b$ (type IV) edges of length  
$
l=2\arccosh\big(2\cos(\alpha/2)/\sqrt{3}\big),
$
pairwise identified by hyperbolic transformations or glide reflections in $\Gamma$.
\end{itemize}
Additionally, the vertices of $D_\Gamma(z)$ are either \emph{ideal vertices} (the $n$ ideal endpoints of edges of type I),  \emph{boundary vertices} (the $2b$ ends of edges of type II) or \emph{interior vertices} (all the rest). The angle of $D_\Gamma(z)$ is zero at the ideal vertices, $\pi/2$ at the boundary vertices and $2\pi/3$ at the interior vertices. All boundary vertices are at a distance $D=\arccosh \left( \sqrt{\cosh(2r_{\chi,n,b})} \right)$ from $z$ and get pairwise identified by $\Gamma$, while all the interior vertices are at a distance $d=\arccosh \left( \mathrm{cotan}  (\alpha/2)/ \sqrt{3} \right) $ from $z$ and  $\Gamma$ identifies them in triples.
\end{theorem}

\begin{proof}
Consider a decomposition of $S$ into equilateral and horocyclic triangles and Saccheri quadrilaterals like the one described in \autoref{theo1}. This decomposition has $n$ horocyclic ideal triangles and $b$ Saccheri quadrilaterals. If $q$ is the number of equilateral triangles, an area computation, together with Gauss-Bonnet Theorem, gives
\begin{equation}\label{eqteo1}
    -2\pi\chi=\text{Area}(S)=q(\pi-3\alpha)+n(\pi-2\beta)+b(\pi-2\gamma)
\end{equation}

\noindent where $\alpha=\alpha(r_{\chi,n,b})$, $\beta=\beta(r_{\chi,n,b})$ and $\gamma=\gamma(r_{\chi,n,b})$ with the notation of \autoref{theo1}. 

From \eqref{eqteo1} we can compute
$$
q(\pi-3\alpha)=\pi(-2\chi-n-b)+2n\beta+2b\gamma
$$
which, using \eqref{deblois}, gives
$$
\begin{array}{rcl}
q(\pi-3\alpha)& = &\pi(-2\chi-n-b)+2\pi-(6-6\chi-3n-3b)\alpha \\
& = & (\pi-3\alpha) (2-2\chi-n-b) 
\end{array}
$$
and therefore the number of equilateral triangles is $q=2-2\chi-n-b$.

Given the quotient map $\pi:\mathbb{D}\longrightarrow \mathbb{D}/\Gamma \simeq S$, we can choose three complete collections of representatives $\mathcal{T}=\{T_1,\ldots, T_q\}$, $\mathcal{S}=\{S_1,\ldots, S_n\}$ and  $\mathcal{Q}=\{Q_1,\ldots, Q_b\}$ of, respectively,  the equilateral triangles, the horocyclic ideal triangles and the Saccheri quadrilaterals of this decomposition, such that each element in  $\mathcal{T} \cup \mathcal{S} \cup \mathcal{Q}$ shares one edge with another of its elements. In particular, $\Omega= T_1\cup \ldots \cup T_q \cup S_1 \cup \ldots \cup S_n \cup Q_1 \cup \ldots \cup  Q_b $ is a hyperbolic polygon.

As our starting decomposition of $S$ has just one vertex in the interior of $S$, all the proper vertices of the polygons in $\mathcal{T} \cup \mathcal{S} \cup \mathcal{Q}$ belong to the same $\Gamma$-orbit, with the exception of the vertices at the bases of the Saccheri quadrilaterals. Indeed, we can distinguish two types of proper vertices in $\Omega$:
\begin{itemize}
 \item The vertices lying at the bases of Saccheri quadrilaterals, that are pairwise identified  by $\Gamma$ (as otherwise the number of boundary components would be smaller than $b$). The sum of the angles in each $\Gamma$-orbit of this kind equals $\pi$.
    \item All the remaining proper vertices, which belong to the same $\Gamma$-orbit. The sum of the corresponding angles is $3q\alpha+2n\beta+2b\gamma=2\pi$ by \eqref{deblois}.
\end{itemize}

In particular the angle of $\Omega$ at any vertex of the second kind, which is at most $q\alpha+n\beta+b\gamma$, cannot exceed $\pi$, hence $\Omega$ is a convex polygon.

It follows from Poincaré's Polygon Theorem that $\Omega$ is a fundamental domain of the NEC group $\Gamma$. The associated tessellation of $\mathbb{D}$ given by $\mathbb{D}=\{\gamma(\Omega):\gamma\in\Gamma\}$ can be re-interpreted as a tessellation by $\Gamma$-translates of the pieces of $\Omega$, namely
$$
\mathbb{D} = \left(\bigcup_{T\in \mathcal{T}}\{\gamma(T):\gamma\in\Gamma\}\right) \bigcup \left(\bigcup_{S\in \mathcal{S}}\{\gamma(S):\gamma\in\Gamma\}\right) \bigcup \left(\bigcup_{Q \in \mathcal{Q}}\{\gamma(Q):\gamma\in\Gamma\}\right).
$$

Let $s=3q=3(2-2\chi-n-b)$ and let $z\in\mathbb{D}$, which we may assume to be the origin, be a representative of the unique  vertex of the original cell  decomposition of $S$  that lies in the interior of $S$. We can now put a number of pieces of this last tessellation together to form a polygon $P'$ that is a complete neighborhood around $z$. More precisely, these pieces that form $P'$ are a collection $\mathcal{T'}$ of $s$ equilateral triangles, a collection $\mathcal{S'}$ of $2n$ horocyclic ideal triangles and a collection $\mathcal{Q'}$ of $2b$ Saccheri quadrilaterals, all them meeting at this point $z$. The number of pieces of each type follows from the number of angles $\alpha$, $\beta$ and $\gamma$ around $\pi(z)$, as the three vertices of the equilateral triangles, the two proper vertices of the horocyclic ideal triangles and the two vertices of angle $\gamma$ in the Saccheri quadrilaterals all project to the same point via $\pi$.

The polygon $P'$ is \emph{centered} at $z$, in the sense that all the vertices of $P'$ that are $\Gamma$-related to $z$ (we will call them  $z$\emph{-vertices}) lie at a distance $2r_{\chi,n,b}$ from $z$. Each vertex of $P'$ is either a $z$-vertex, an ideal vertex at the boundary of $\mathbb{D}$, or an endpoint of the base of a Saccheri quadrilateral. Note that there is another obvious kind of vertices in the tessellation of $\mathbb{D}$ at a distance $2r_{\chi,n,b}$ from $z$: the vertices $\gamma(z)\in\mathbb{D}$, where $\gamma\in\Gamma$ is the hyperbolic reflection in the base of a Saccheri quadrilateral in the construction of $P'$. These points will be called \emph{generalized} $z$-\emph{vertices}.

Let $\{T'_{1}, \ldots, T'_{q}\}\subset \mathcal{T'}$, $\{S'_{1}, \ldots, S'_{n}\}\subset \mathcal{S}'$ and $\{Q'_{1}, \ldots , Q'_{b}\}\subset \mathcal{Q}'$ be a complete set of representatives of each of the three type of pieces of $P'$.

\

\textbf{Claim 1:} \emph{Every horocyclic ideal triangle $S'_k\in \mathcal{S}'$ shares a non-compact edge with another horocyclic  triangle $\widetilde{S}'_k\in \mathcal{S}'$, which will be called the \emph{twin} of $S'_k$. Accordingly, every Saccheri quadrilateral $Q'_l\in \mathcal{Q}'$ shares a leg with another Saccheri quadrilateral $\widetilde{Q}'_l\in\mathcal{Q}'$, which will be called the \emph{twin} of $Q'_l$.}

\

For any given $k\in \{1,\ldots, n\}$, there exists a parabolic isometry $\gamma_k\in\Gamma$ that identifies the non-compact edge of $S'_k$ that has $z$ as an endpoint with the other non-compact edge of $S'_k$ (if it was not the case there would be less than $n$ cusps in the quotient surface $S=\mathbb{D}/\Gamma$). Notice also  that $\gamma_k^{-1}(S'_k)$ is a horocyclic ideal triangle that meets $z$, necessarily different from $S_i'$ for all $i \in \{1,\ldots, n\}$, and obviously $S'_k$ and $\gamma_k^{-1}(S'_k)=\widetilde{S}'_k\in\mathcal{S}'$ share one edge. We can mimic this argument for Saccheri quadrilaterals due to the existence of a hyperbolic isometry $\eta_l\in\Gamma$ that identifies the leg of $Q'_l$ that has $z$ as an endpoint with the other leg of $Q'_l$, where $l\in\{1,\ldots,b\}$. We have thus proved Claim 1. 

As a consequence, observe that each $\Gamma$-orbit of horocyclic ideal triangles has two elements in $P'$, namely $S'_k$ and its twin $\widetilde{S}'_k$, for $1\le k\le n$. Correspondingly, each $\Gamma$-orbit of Saccheri quadrilaterals has two representatives in $P'$, which are $Q'_l$ and its twin $\widetilde{Q}'_l$, for $1\le l \le b$.

\

\textbf{Claim 2:} \emph{Every horocyclic ideal triangle in $\mathcal{S}'$ shares its compact edge with an equilateral triangle in $\mathcal{T}'$. Analogously, the summit of any Saccheri quadrilateral in $\mathcal{Q}'$ is also an edge of some equilateral triangle in $\mathcal{T}'$.}

 \ 
 
If the statement in Claim 2 was not true, by Claim 1 we would have one of the situations exhibited in \autoref{fig:forbidden}. In the first case, $\gamma_j^{-1}(S'_k)$ would be a horocyclic ideal triangle in the tessellation,  with vertex at $z$, belonging to the same $\Gamma$-orbit as $S'_k$, so we conclude that $$\gamma_j^{-1}(S'_k)=\widetilde{S}'_k=\gamma_k^{-1}(S'_k).$$

In turn, this implies that $\gamma_j=\gamma_k$ and necessarily $j=k$, which is a contradiction. Similar arguments apply to the other two cases in Claim 2.

\begin{figure}[!htbp]
\begin{subfigure}{0.33\linewidth}
  \centering
  \includegraphics[width=0.7\textwidth]{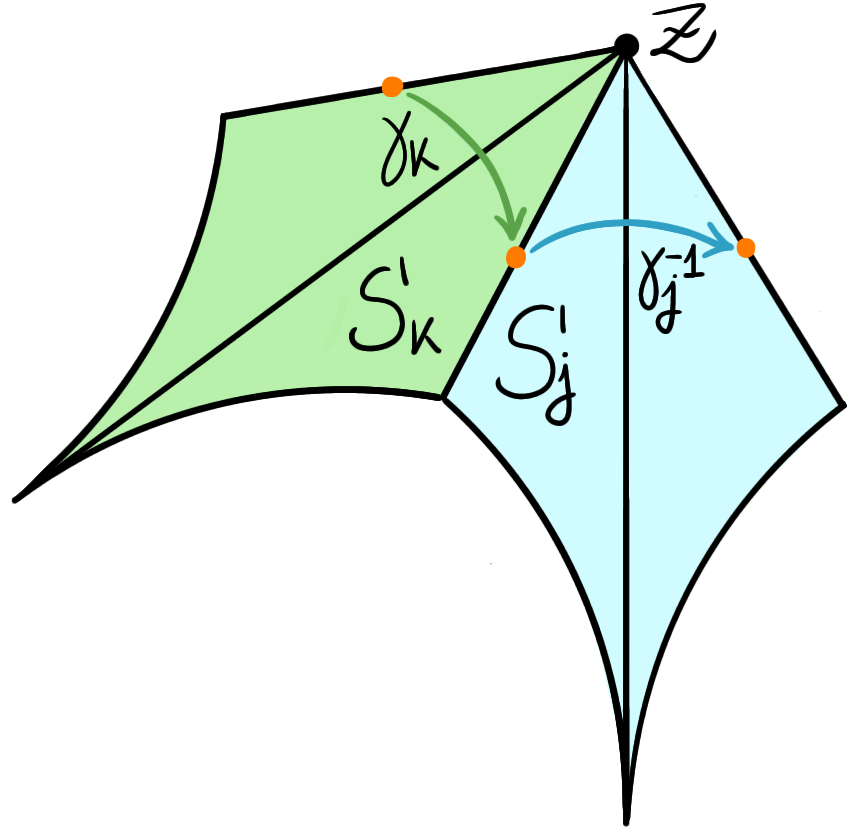}
\end{subfigure}%
\begin{subfigure}{0.33\linewidth}
  \centering
  \includegraphics[width=0.78\textwidth]{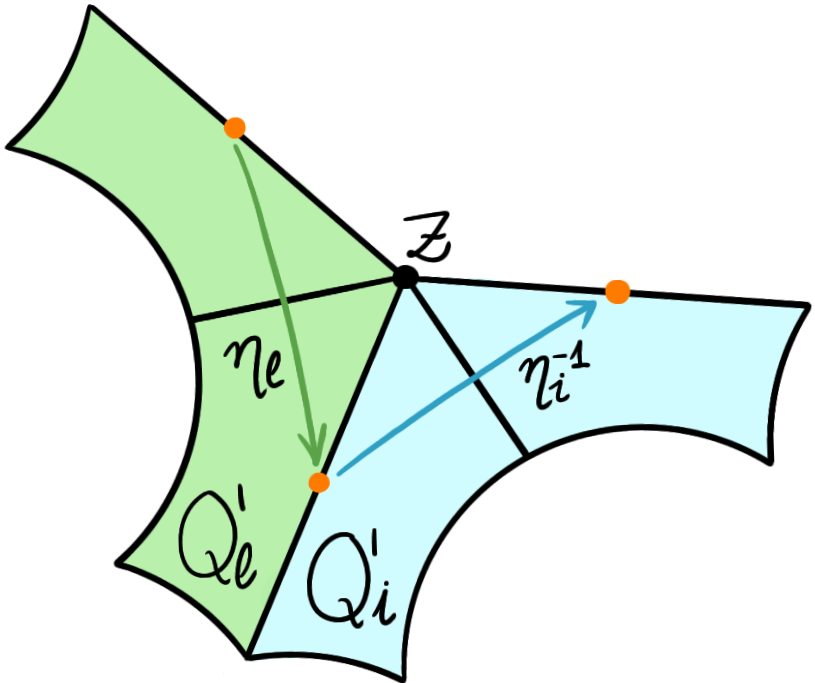}
\end{subfigure}
\begin{subfigure}{0.33\linewidth}
  \centering
  \includegraphics[width=0.95\textwidth]{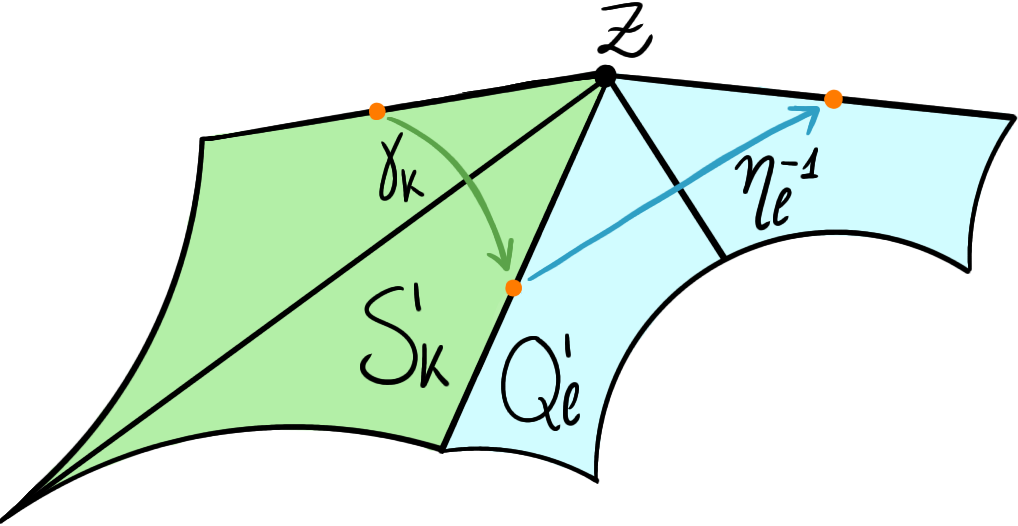}
\end{subfigure}
\caption{Three configurations that cannot occur in $P'$.}
\label{fig:forbidden}
\end{figure}

Consider the set $\mathscr{B}$ given by the perpendicular bisectors of the geodesic rays that connect each $z$-vertex and each generalized $z$-vertex with the point $z$. Note that the second kind of bisectors are simply  the complete geodesics that contain the base of the Saccheri quadrilaterals in $\mathcal{Q}'$. The location of the rest 
of these bisectors inside the equilateral triangles, the Saccheri quadrilaterals and the horocyclic ideal triangles in $P'$, is shown in \autoref{fig:bisectors}. Two bisectors meet at the barycenter of every equilateral triangle, with angle  $2\pi/3$, while there is only one bisector segment inside every Saccheri quadrilateral (namely, the common perpendicular to the base and the summit) or horocyclic ideal triangle (namely, the perpendicular bisector of its compact edge). Note that one ideal endpoint of the latter is precisely the ideal vertex of the horocyclic triangle, so the bisector segments inside two twin horocyclic ideal triangles share their ideal endpoint, as in \autoref{fig:bisectors}.

\begin{figure}[!htbp]
\begin{subfigure}{.315\textwidth}
  \centering
  \includegraphics[width=.58\linewidth]{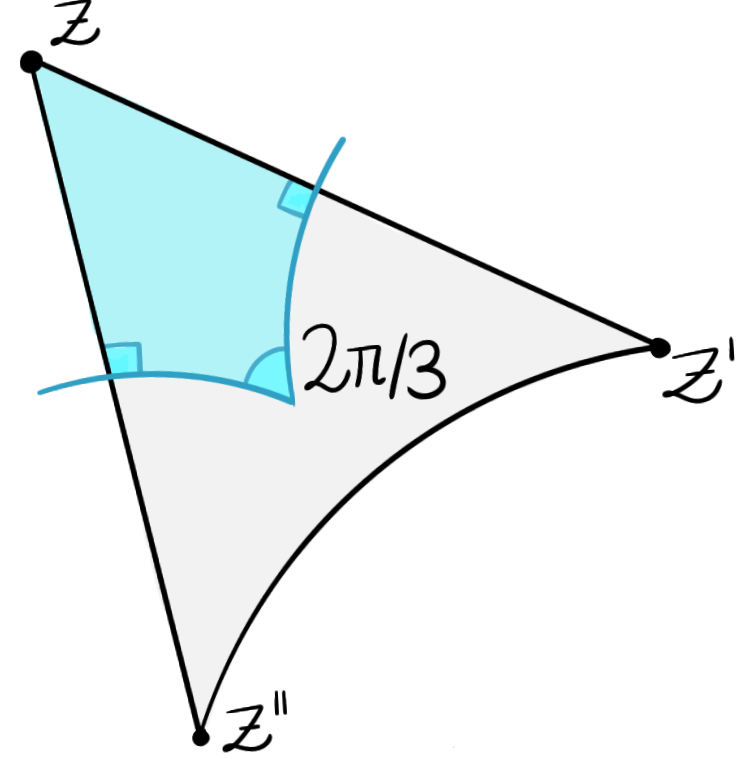}
  \label{fig:sfig1}
\end{subfigure}%
\begin{subfigure}{.362\textwidth}
  \centering
  \includegraphics[width=0.9\linewidth]{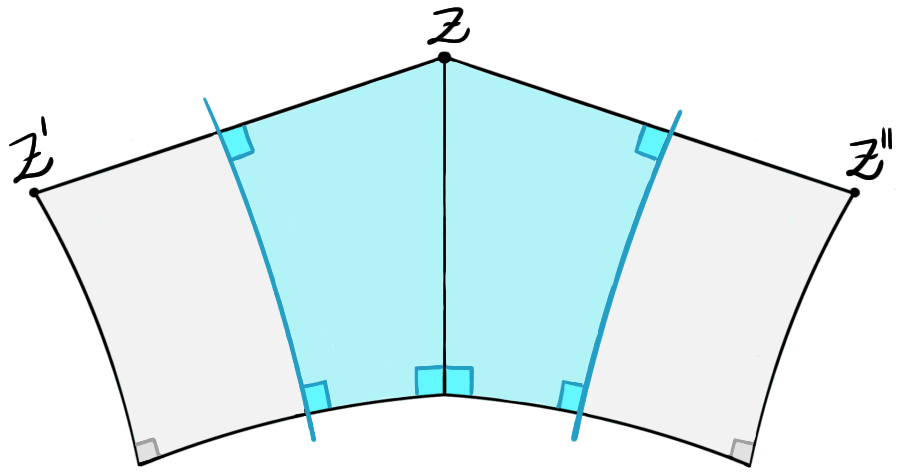}
  \label{fig:sfig2}
\end{subfigure}
\begin{subfigure}{.315\textwidth}
  \centering
  \includegraphics[width=.6\linewidth]{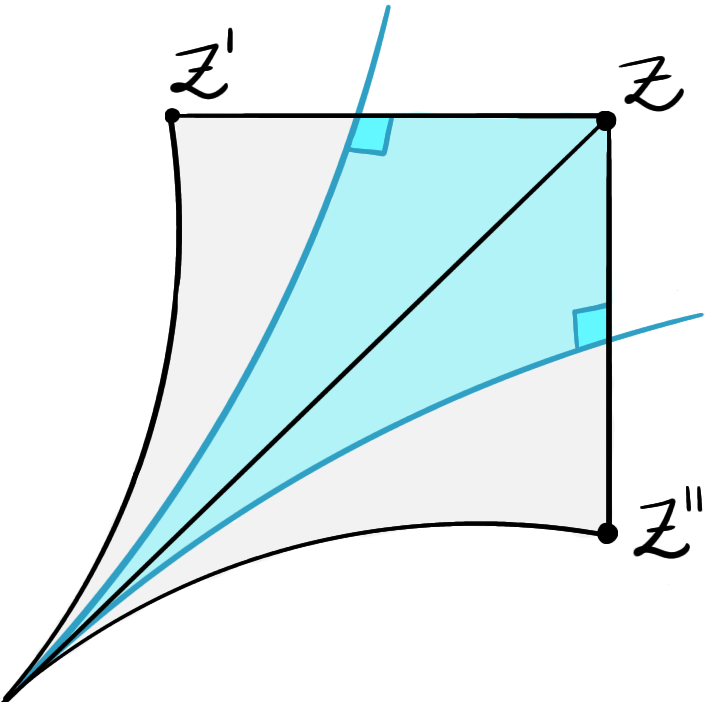}
  \label{fig:sfig3}
\end{subfigure}%
\caption{{Perpendicular bisectors inside an equilateral triangle, two twin Saccheri quadrilaterals and two twin horocyclic ideal triangles.}}
\label{fig:bisectors}
\end{figure}

Each of the perpendicular bisectors in $\mathscr{B}$ defines a hyperbolic half-plane containing the point $z$. Let $P$ be the convex hyperbolic polygon obtained by intersecting these half-planes. In other words, $P$ is the Voronoi $2$-cell at $z$ associated to the finite set $\mathcal{Z}$ given by the point $z$, the $z$-vertices and the generalized $z$-vertices. This set $\mathcal{Z}$ is contained in the $\Gamma$-orbit of $z$, $\pi^{-1}(\pi(z))=\{\gamma(z):\gamma\in\Gamma\}$, so it follows that $D_\Gamma(z) \subset P$, where $D_\Gamma(z)$ stands for the Dirichlet fundamental domain of $\Gamma$ centered at $z$. 

Consider the decomposition of $P$ determined by the following collection of lines:

\begin{itemize}
    \item The geodesic lines connecting $z$ to each vertex of $P$ of angle $2\pi/3$ (which are, by construction, the bisectors of these angles).
Hyperbolic trigonometry shows that the distance from $z$ to any of these vertices equals 
$d=\arccosh \left( \mathrm{cotan}  (\alpha/2)/ \sqrt{3} \right) $.

   \item The geodesic lines connecting $z$ to each vertex of angle zero (ideal vertices). These are simply the common edges to all the pairs of twin horocyclic ideal triangles.
    \item The geodesic lines  connecting $z$ to the midpoint of each edge with two right angles (these are the common legs to the pairs of twin Saccheri quadrilaterals in $P'$, as in the central image of \autoref{fig:bisectors}).
\end{itemize}

By Claim 2, we know that the decomposition of $P$ defined by the preceding collection of lines divides $P$ into pieces like those depicted in \autoref{fig:decomposition}.

\begin{figure}[!htbp]
\begin{subfigure}{.33\textwidth}
  \centering
\includegraphics[width=.57\linewidth]{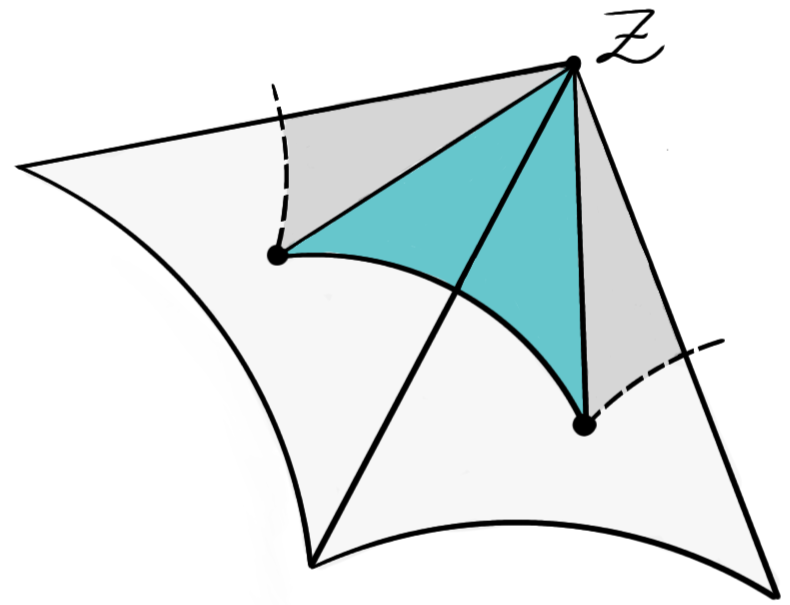}
\end{subfigure}%
\begin{subfigure}{.34\textwidth}
  \centering
\includegraphics[width=.57\linewidth]{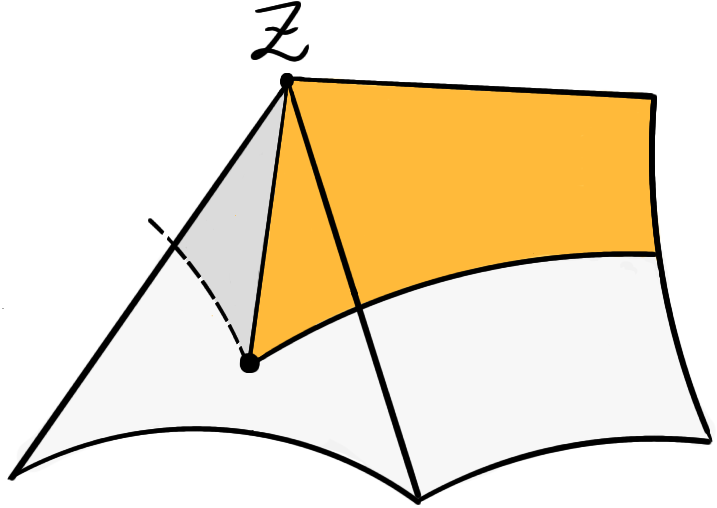}
\end{subfigure}
\begin{subfigure}{.32\textwidth}
  \centering
\includegraphics[width=.84\linewidth]{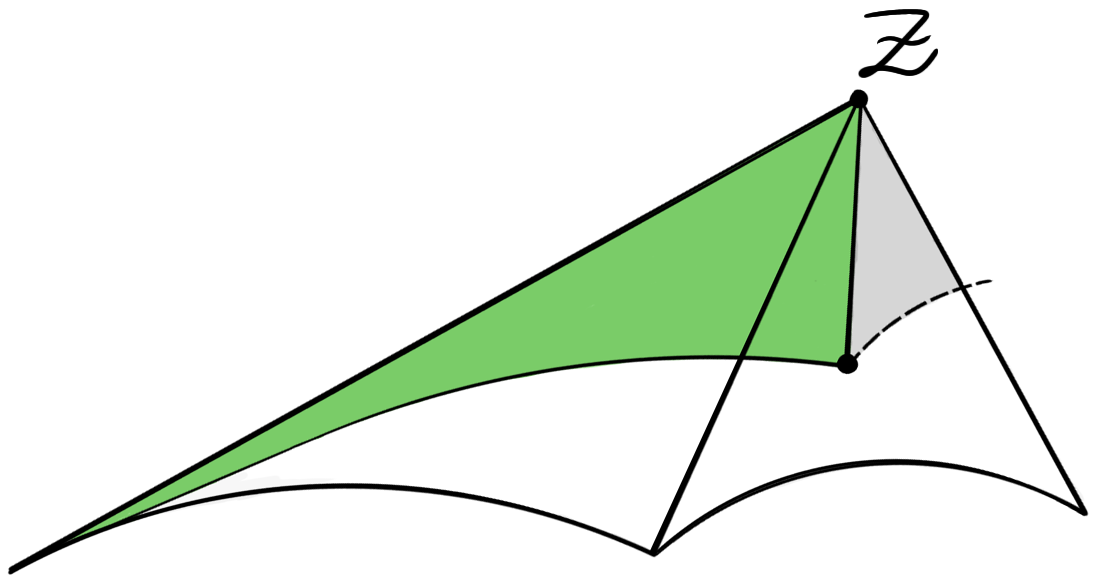}
\end{subfigure}%
\caption{{The three different pieces of the decomposition of $P$.}}
\label{fig:decomposition}
\end{figure}

Thus, $P$ gets decomposed into $2n$ ideal triangles, $2b$ quadrilaterals and $s-n-b$ compact triangles. Indeed, the number of compact triangles here should coincide with the number of $z$-vertices that do not belong either to a horocyclic triangle or to a Saccheri quadrilateral in $P'$, and there are $s-n-b$ such vertices due to Claim 2. Moreover, the angles of these pieces are, by construction, as shown in \autoref{fig:theor_angles}.

\begin{figure}[!htbp]
\begin{subfigure}{.49\linewidth}
	\centering
\includegraphics[width=0.83\linewidth,draft=False]{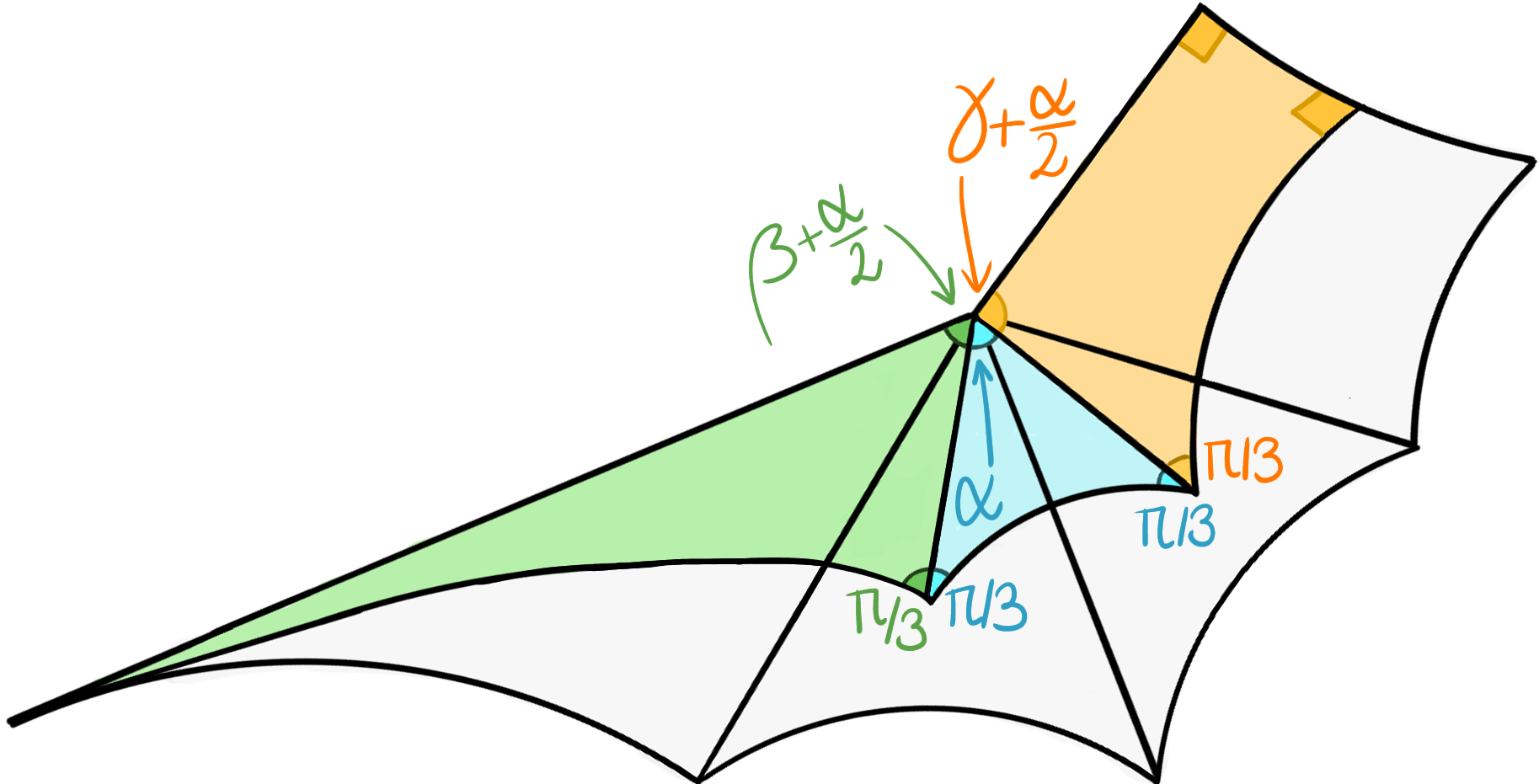}
\end{subfigure}%
\begin{subfigure}{.49\textwidth}
	\centering
\includegraphics[width=0.83\linewidth,draft=False]{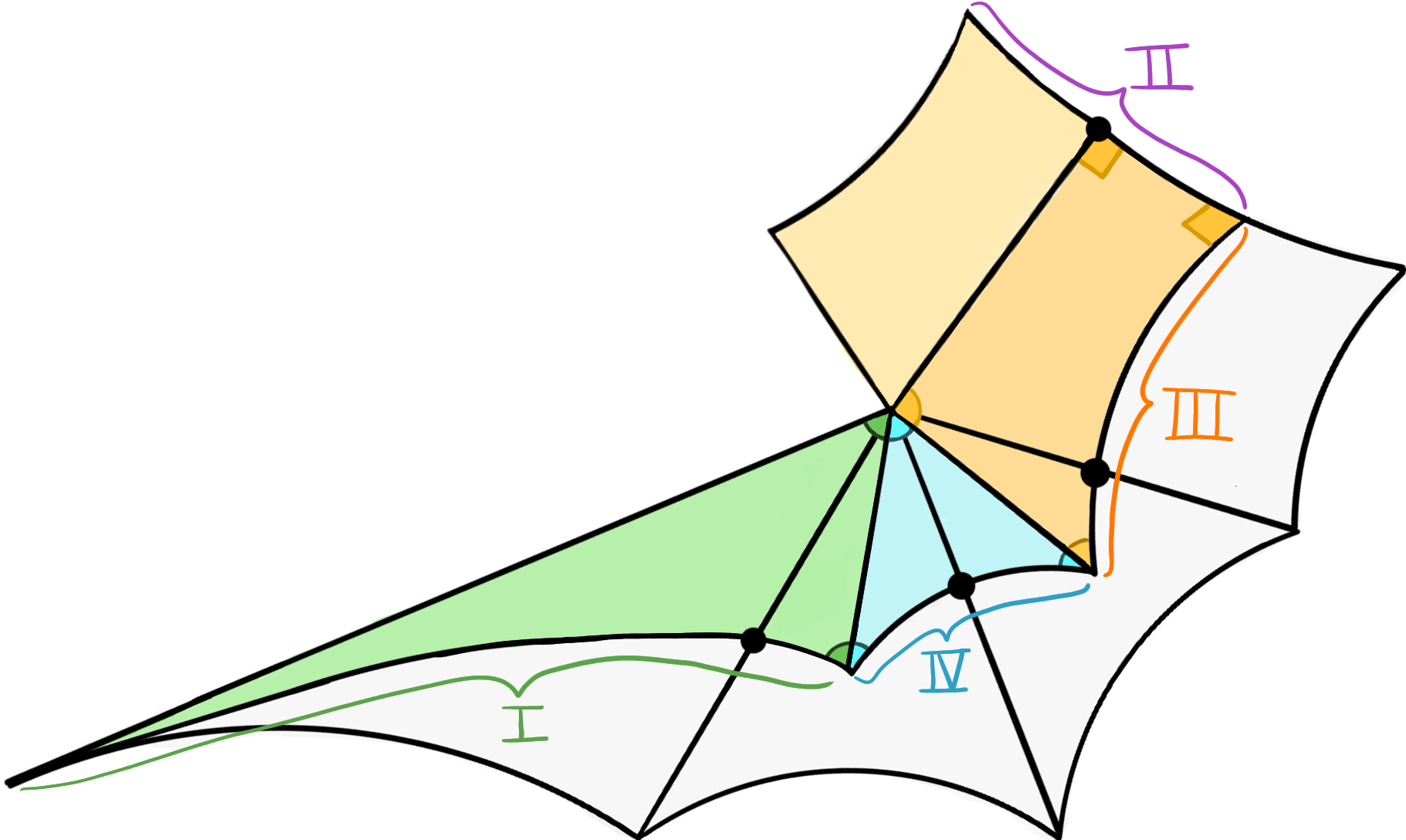}
\end{subfigure}
	\caption{On the left, the angles of the three kind of pieces in the decomposition of $P$. On the right, the four  kind of edges of $P$ and the points at its boundary lying at distance $r_{\chi,n,b}$ of $z$.}
 \label{fig:theor_angles}
\end{figure}

A direct computation of the total area of these pieces, together with formula in \eqref{deblois} and  Gauss-Bonnet Theorem shows that
\begin{align*}
    \text{Area}(P)&=(s-n-b)\left(\pi-\frac{2\pi}{3}-\alpha \right)+2n\left(\pi-\frac{\pi}{3}-\beta-\frac{\alpha}{2}\right)+2b \left(2\pi-\frac{4\pi}{3}-\gamma-\frac{\alpha}{2} \right)\\
    & = (6-6\chi -4n-4b) \left( \frac{\pi}{3}-\alpha \right)+2(n+b) \left( \frac{2\pi}{3}-\frac{\alpha}{2} \right) + (6-6\chi -3n-3b)\alpha -2\pi \\
    &=-2\pi\chi=\text{Area}(S)
\end{align*}

We claim that there are no $\Gamma$-identified points in the interior of $P$. To prove this, take a small disc centered at $z$, and notice that each of the pieces of $P$ is determined by a different circle sector around $z$, as all the pieces meet at $z$.  If $p, p'$ were two different points in the interior of $P$ belonging to two pieces $\rho$, $\rho'$ of this decomposition and $\eta(p)=p'$ for some $\eta\in\Gamma$, then $\eta$ would identify $\rho$ and $\rho'$, hence the corresponding circle sectors around $z$, which is a contradiction. 

Putting together that no pair of points in the interior of $P$ are $\Gamma$-identified, that  $D_\Gamma(z)\subset P$ and that $\text{Area}(P)=\text{Area}(S)=\text{Area}\big(D_\Gamma(z)\big)$, we conclude that in fact $P$ agrees with the Dirichlet fundamental domain $D_\Gamma(z)$. By the aforementioned decomposition into pieces, the number of edges of $D_\Gamma(z)$ equals
$$m=(s-n-b)+2n+3b=s+n+2b=6-6\chi-2n-b.$$

In order to complete the proof of all the statements in \autoref{theo2}, we need to compute the length of the four different kinds of edges of $D_\Gamma(z)$, which are represented in 
\autoref{fig:theor_angles}.

By construction, there are exactly $2n$ non-compact edges, the edges of type I in \autoref{fig:theor_angles}. Denoting by $l$ the length of the edges of type IV, by the hyperbolic law of cosines we have
$$
\cosh(l/2)=2\cos(\alpha/2)/\sqrt{3}.
$$

There are $2b$ edges of type III, one for each quadrilateral in our decomposition of $D_\Gamma(z)$. If we denote by $l^*$ the length of this kind of edges, as $l^*-l/2$ is the length of the common perpendicular to the base and the summit of a Saccheri quadrilateral in $\mathcal{Q}'$ (see \autoref{fig:bisectors}), by applying [\cite{beardon}, Theorem 7.17.1] and \autoref{angles} we see that
$$
\cosh(l^*-l/2)=\cosh(r_{\chi,n,b})\sin(\gamma)=\sqrt{2\cos(\alpha)}.
$$

Finally, the $b$ edges of type II, which correspond to the $b$ boundary components of $S$, have length $L=2\arcsinh(\tanh(r_{\chi,n,b}))$  by \autoref{theo1}. The distance  from $z$ to any vertex of this kind of edges equals
$$D=\arccosh \left( \cosh (r_{\chi,n,b}) \cosh (L/2) \right)=\arccosh \left( \sqrt{\cosh(2r_{\chi,n,b})}\right)$$ by a straightforward hyperbolic trigonometry computation.
\end{proof}

\begin{remark} \label{rem:6edges}
From the proof we see that there are exactly $m=6-6\chi-2n-b$ points in the boundary of the polygon $D_\Gamma(z)$  at minimal distance $r_{\chi,n,b}$ from $z$ (one per edge). \autoref{fig:theor_angles} shows the location of these points for each type of edge. 

\end{remark}

\begin{definition} \label{def:extrdiscconf}
A hyperbolic polygon $P$ together with a pattern of side-pairing identifications satisfying all the conditions in \autoref{theo2} is called an \emph{extremal disc configuration based on} $P$. Two extremal disc configurations based on the polygons $P_1$ and $P_2$
 are considered \emph{equivalent} if there exists an isometry $\tau$ of the hyperbolic plane that maps $P_1$ into $P_2$ and conjugates its side-pairing pattern into the other. We refer to the equivalence classes under this definition as \emph{extremal disc configurations}. 
\end{definition}

\begin{example}\label{ex_sph_three}
When $S$ is a non-orientable surface we have $\chi=2-g-n-b$, which implies $m=6g+4n+5b-6$. Using the fact that  $\chi<0$ one can easily see that  the number $m$ of edges of $D_\Gamma(z)$ is strictly greater than $6+b$. On the other hand, for the orientable case we have $\chi=2-2g-n-b<0$, which implies $m=12g+4n+5b-6\geq 6+b$, with equality if and only if $g=0$ and $n+b=3$. 

There are only four possibilities for $m$ attending the minimal possible value $m=6+b$ for (necessarily orientable) extremal surfaces with $b$ boundary components, corresponding to the data $(g,n,b)=(0,3-b,b)$ and $b\le 3$. In each case, the shape of the polygon $D_\Gamma(z)$ is completely determined, up to isometry, by \autoref{theo2} and also the side-pairing pattern induced by the corresponding NEC group is fully determined (see \autoref{fig:casesg=0_rho=3}), so there is only one extremal disc configuration corresponding to these triplets $(g,n,b)$. We note in passing that we can get $m=6$ only for the case of the three times punctured sphere, which accounts for the only extremal disc configuration based on a polygon with 6 edges. We will come back to this extremal surface in \autoref{ex:3-punctured_sphere}.

\begin{figure}[!htbp]
\begin{subfigure}{.24\textwidth}
  \centering
\includegraphics[width=0.99\linewidth]{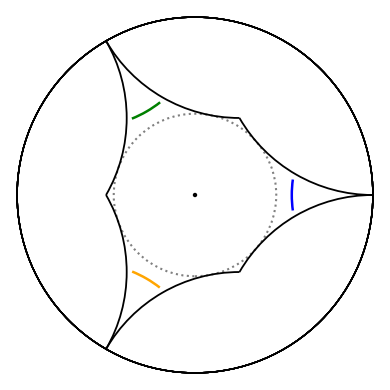}
\end{subfigure}
\begin{subfigure}{.24\textwidth}
\includegraphics[width=0.99\linewidth]{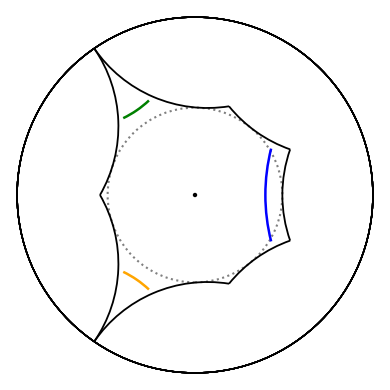}
\end{subfigure}
\begin{subfigure}{.24\textwidth}
\includegraphics[width=0.99\linewidth]{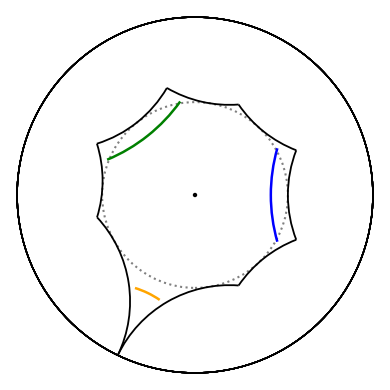}
\end{subfigure}
\begin{subfigure}{.24\textwidth}
\includegraphics[width=0.99\linewidth]{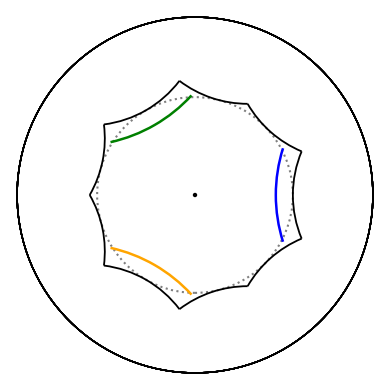}
\end{subfigure}%
\caption{The four extremal surfaces of genus 0 with $n+b=3$.}
\label{fig:casesg=0_rho=3}
\end{figure}    

From another point of view, these four extremal disc configurations are also characterized as those configurations based on polygons without edges of type IV, since the condition $6-6\chi-4n-4b=0$ cannot hold for a hyperbolic non-orientable surface and yields precisely $g=0$ and $n+b=3$ in the orientable case.

\end{example}

 This example motivates the following definition:

\begin{definition}
The \emph{number of extended boundary components}
of a hyperbolic surface with Euler characteristic $\chi$, $b$ boundary components and $n$ cusps is the sum $\rho=n+b$.
\end{definition}

In fact, considering cusps and proper boundary components as a whole is a natural idea, since one can think of cusps as boundary components of length zero. 
Obviously, the $\rho+1$ different pairs $(n,b)=(\rho-b,b)$ for $0\le b\le \rho$ can happen for a given $\rho$. In    \autoref{fig:casesg=0_rho=3} we see all the extremal disc configurations for the case $g=0$ and $\rho=3$. This is a particular situation in which there is a single possible configuration for each pair  $(n,b)=(\rho-b,b)$, but this does not happen in general (see \autoref{fig:fig_12_noori_n2b1}).

In the next section we will make an extensive use of this idea of considering cusps and boundary components altogether, as from a combinatorial point of view both features play a similar role. However, this is not the case from a metric point of view.    \autoref{fig:casesg=0_rho=3} already suggests that for $\chi$ and $\rho$ given, the actual number of cusps affect the size  of extremal discs, as more cusps seem to produce smaller extremal radii. This is indeed the case in general:  

\begin{lemma} \label{le:morecuspssmaller}
    For fixed integer numbers $\chi<0$, $\rho\geq 0$ and $n, n'\in \{0,\ldots , \rho\}$ with $ n< n'$, the inequality $r_{\chi,n,\rho-n}> r_{\chi,n',\rho-n'}$ holds. In particular, the case of $\rho$ cusps and empty boundary gives the smallest radius of extremal discs inside 
  surfaces with Euler characteristic $\chi$ and $\rho$ extended boundary components.
\end{lemma}

\begin{proof}
Recall that $r_{\chi,n,b}$ is the unique positive solution to (\ref{deblois}), i.e. the preimage of $2\pi$ under the decreasing function $f_{\chi,n,b}(r)=\left(6-6\chi-3n-3b\right)\alpha(r)+2n\beta(r)+2b\gamma(r)$.
For $n<n'$, by denoting $b=\rho-n$ and $b'=\rho-n'$, a direct computation gives 
\begin{align*}f_{\chi,n',b'}(r_{\chi,n,b})&=\left(6-6\chi-3\rho\right)\alpha(r_{\chi,n,b})+2n'\beta(r_{\chi,n,b})+2b'\gamma(r_{\chi,n,b})\\
& = f_{\chi,n,b} (r_{\chi,n,b}) +2(n'-n) \beta(r_{\chi,n,b}) +
2(b'-b) \gamma(r_{\chi,n,b}) \\
&< f_{\chi,n,b} (r_{\chi,n,b}) +2(n'-n+b'-b) \gamma(r_{\chi,n,b})  \\
&=f_{\chi,n,b} (r_{\chi,n,b})=2\pi=f_{\chi,n',b'}(r_{\chi,n',b'})
\end{align*}
where we have used that $\beta<\gamma$ (see \autoref{angles}).
\end{proof}

\begin{proposition}\label{minimum_injrad}
The radius of an extremal disc inside any  hyperbolic surface different from the three-times punctured sphere is strictly greater than $r_{-1,3,0}$.
\end{proposition}

\begin{proof}
Extremal radii form a set of non-negative numbers, so it has a well-defined infimum. 
By the previous lemma, this infimum is in fact the infimum of the set $\{r_{\chi,\rho,0}\}$,  where $(\chi, \rho)$ run through all admissible data pairs. On the other hand, an argument similar to the one in the proof of \autoref{le:morecuspssmaller}  shows that for $\rho_0$ given we have $r_{\chi',\rho_0,0}>r_{\chi,\rho_0,0}$ if $\chi'<\chi$. 
 Now, the largest possible Euler characteristic for $\rho \ge 3$ is $\chi=2-\rho$ (corresponding to a sphere with $\rho$ punctures), and $\inf \{r_{\chi,n,b}\}=\inf \mathcal{R}_i$ with $\mathcal{R}_i=\{ r_{-1,0,0}, r_{-1,1,0}, r_{-1,2,0}\} \cup \{ r_{2-\rho,\rho,0}: \rho \ge 3 \}$.
 
 For $\rho\in \{0,1,2\}$,
$ f_{-1,\rho,0}(r_{-1,3,0})=3(4-\rho)\alpha(r_{-1,3,0})+2\rho\beta(r_{-1,3,0})>2\pi$ as $3\alpha>2\beta$ by \autoref{angles}, so $r_{-1,3,0}<r_{-1,\rho,0}$. On the other hand, for every $\rho > 3$ we have $r_{2-\rho,\rho,0}>r_{-1,3,0}$ since  
 $f_{2-\rho,\rho,0}(r_{-1,3,0})=3(\rho-2)\alpha(r_{-1,3,0})+2\rho\beta(r_{-1,3,0})>f_{-1,3,0}(r_{-1,3,0})=2\pi$.
\end{proof}

\begin{remark}\label{bound_ideal_triangles}
 We note in passing that, as a consequence, the angles $\alpha(r_{\chi,n,b})$, $\beta(r_{\chi,n,b})$ and $\gamma(r_{\chi,n,b})$ attain their maximal values when $(\chi,n,b)=(-1,3,0)$. In particular, the amount $\beta(r_{\chi,n,b})+\alpha(r_{\chi,n,b})/2$ is bounded above by $\pi/3$, which is the value corresponding to the three-times punctured sphere.
\end{remark}

\section{A first application: counting extremal disc configurations}

\label{sec:extremaldiscconfigurations}

Let $N_{\chi,n,b}$ be the number of extremal disc configurations with Euler characteristic $\chi$, $n$ cusps and $b$ boundary components. This number is related, but not necessarily identical, to the number  $\Sigma_{\chi,n,b}$ of equivalence classes of hyperbolic surfaces for the same datum $(\chi,n,b)$.
The point here is that, in principle, a given surface could  present extremal discs having different extremal disc configurations. In fact, this is something that actually happens in the case of extremal $k$-packings embedded in closed hyperbolic surfaces when $k>1$ (see \cite{girondo_2017}). This is not the case for single extremal discs in closed surfaces (see, for instance, \cite{Girondo_Gonzalez-Diez_2002_2}). Moreover, the work in \cite{beauchamp2017} regarding extremal punctured spheres and our own experience along the composition of this paper support the  conjecture that counting extremal disc configurations or counting extremal surfaces may turn out to be equivalent problems.

Constructing all extremal disc configurations amounts to building up first all the possible polygons like those described in \autoref{theo2} and, for each such polygon, defining all the coherent sets of side-pairing identifications with the demanded properties. This strategy, together with a method that we call the \emph{EBF-correspondence}, yields the following bounds for the number of inequivalent extremal disc configurations:

\begin{theorem}\label{count}
 For any triplet $(\chi,n,b)$ of integers such that $\chi<0$ and $n,b\geq 0$, the inequality $$\# C_{M,\rho}\le N_{\chi,n,b}\le \displaystyle\binom{\rho}{b} \# C_{M,\rho}$$ holds,  
 where $\rho=n+b$, $M=6-6\chi-2\rho$ and $C_{M,\rho}$ denotes the set of conjugacy classes of subgroups of index $2M$ of the extended triangle group $\Delta^{\pm}(2,3,M)$ with $\rho$ non-conjugate elliptic elements, all of them of order 3. In the particular cases $n=0$ or $b=0$ the identity $N_{\chi,n,b}= \#C_{M,\rho}$ is satisfied.
\end{theorem}

\begin{proof}
We start by pointing out that, according to \autoref{ex_sph_three},  in the four particular cases when $\chi=-1$ and $\rho=n+b=3$,  we have $N_{\chi,n,b}=1 $ and the claim is obviously true. In particular, $\# C_{6,3}=1$, as the Euclidean extended triangle group $\Delta^{\pm}(2,3,6)$ contains one single class of subgroups of index $12$ with three non-conjugate generators, order $3$ rotations whose fixed points are  alternate vertices of a regular hexagon. 

In what follows, we assume that we are not in any of these special cases. Our proof relies on the construction of a surjective map
from the set of extremal disc configurations of data $(\chi,n,b)$ to a set of hyperbolic orbifolds with Euler characteristic $\chi+\rho/3$ and $\rho$ cone points of order $3$ whose uniformizing group is contained in a certain extended triangle group.

Let $S\simeq\mathbb{D}/\Gamma$ be an extremal surface with Euler characteristic $\chi$, $n$ cusps and $b$ boundary components,  $z\in\mathbb{D}$ be a preimage of an extremal disc center in $S$ and $D_\Gamma(z)$ be
 the Dirichlet domain of $\Gamma$ centered at $z$, which is a hyperbolic $m$-gon for $m=6-6\chi-2n-b$ according to \autoref{theo2}. That is, we start with some extremal disc configuration that has $S$ as the underlying hyperbolic surface. Let $M=m-b$, which, according to \autoref{ex_sph_three}, is greater than 6, since we have ruled out the four cases for which $M$ equals 6. In particular, 
we can assume that there exists a regular $M$-sided hyperbolic polygon $Q_M$ with angle $2\pi/3$.

 Label counterclockwise $q_1,\ldots,q_M$ the edges of $Q_M$, and $d_1,\ldots,d_m$ the edges of $D_\Gamma(z)$, where $d_1$ is an edge of type IV (note that the existence of such an edge is guaranteed as we are assuming that $S$ differ from the surfaces studied in \autoref{ex_sph_three}). For $1\le j \le m $ define  
$$
b_j=\#\{d_l: 1\leq l\leq j \text{ and } d_l \text{ corresponds to a boundary component of } S\}
$$
and let $\varphi: \{1, \ldots, m\} \to \{1, \ldots, M \}$ be the surjective map defined by $\varphi(j)=j-b_j$.

The identification pattern of the edges of $D_\Gamma(z)$ defined by $\Gamma$ allows us to construct a natural subgroup  $\Lambda_{\Gamma,z} < \Delta^{\pm}(2,3,M)$ generated by the elements that map $q_{\varphi(k)}$ to $q_{\varphi(l)}$ whenever there exists an isometry of $\Gamma$ mapping $d_k$ to $d_l$ for $k\neq l$, and keeping the condition of being orientation preserving/reversing under this correspondence from $\Gamma$ to $\Delta^{\pm}(2,3,M)$.

Notice that $\Lambda_{\Gamma,z} $ is not torsion-free, unless $\rho=0$, as it contains as many elliptic elements of order three as the number $\rho$ of extended boundary components of $S$. We call the fixed points of these $\rho$ elliptic elements the \emph{marked vertices} of $Q_M$: there are no two consecutive such vertices in $Q_M$ due to Claim 2 in the proof of \autoref{theo2}.

The sum of the angles at each $\Lambda_{\Gamma,z}$-orbit of non-marked vertices is $2\pi$, as these vertices are identified in triplets as in $D_\Gamma(z)$, while the sum of the  angles at each $\Lambda_{\Gamma,z}$-orbit of marked vertices is $2\pi/3$. By Poincaré's Polygon Theorem, $\Lambda_{\Gamma,z}$ is a NEC group with fundamental domain $Q_M$. Actually, we see that $\Lambda_{\Gamma,z} $ is a subgroup of index $2M$ of the extended triangle group $\Delta^{\pm}(2,3,M)$ that has $\rho$ different conjugacy classes of elliptic isometries of order 3.

This rule $D_{\Gamma}(z) \to \Lambda_{\Gamma, z}$ induces a correspondence

\vspace{0.3em}

\noindent
\begin{minipage}{0.46\textwidth}
\fbox{%
\parbox{0.95\textwidth}{
 Extremal disc configurations of surfaces with Euler characteristic $\chi$, $n$ cusps and $b$ boundary components.
  }%
}
\end{minipage}
\begin{minipage}{0.075\textwidth}
 \;$\xrightarrow{EBF}$
\end{minipage}
\begin{minipage}{0.46\textwidth}
\fbox{%
\parbox{0.95\textwidth}{
 Conjugacy classes of subgroups of index $2M$ of $\Delta^{\pm}(2,3,M)$ with $\rho$ inequivalent elliptic elements, all of them of order 3, where $\rho=b+n$ and $M=6-6\chi-2\rho$.
  }%
}
\end{minipage}
\

\vspace{0.3em}

\noindent that we call EBF as an acronym of \emph{extended boundary filling}. 

For example, one can easily check that the NEC group described by the generators on the left-hand side of \autoref{fig:counting} is mapped by the EBF-correspondence to the group determined by the generators on the right-hand side.

\begin{figure}[!hbtp]
\begin{subfigure}{.4\textwidth}
  \centering
\includegraphics[width=.77\linewidth, draft=False]{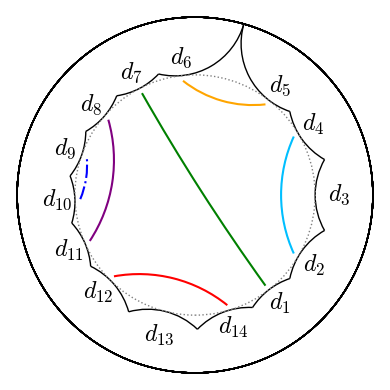}
\end{subfigure}%
\begin{subfigure}{.4\textwidth}
  \centering
\includegraphics[width=.79\linewidth, draft=False]{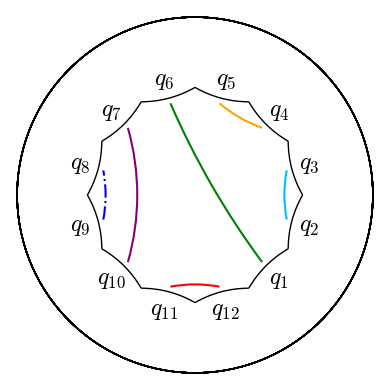}
\end{subfigure}
\caption{On the right side, the EBF-image of the extremal disc configuration on the left. The side-pairings $d_2-d_4$, $d_5-d_6$ and $d_{12}-d_{14}$ correspond to the order three elliptic elements identifying the edges $q_2-q_3$, $q_4-q_5$ and $q_{11}-q_{12}$ respectively. The edges $d_9$ and $d_{10}$ (resp. $q_8$ and $q_9$) are identified by an orientation reversing element, while the rest of identifications preserve the orientation.}
\label{fig:counting}
\end{figure}

We claim that the EBF-correspondence is surjective and the cardinality of the preimage of a given $\Lambda < \Delta^{\pm}(2,3,M)$ is at most $\binom{\rho}{b}$.

First, note that  if $\Lambda$ is a subgroup of index $2M$ of $\Delta^{\pm}(2,3,M)$ with exactly $\rho$ conjugacy classes of elliptic elements, all of order 3, then $\Lambda$ admits $Q_M$ as a fundamental domain. We call a vertex of $Q_M$ \emph{marked} if it is fixed by some elliptic element of $\Lambda$, and \emph{regular} if not. There are $\rho$ marked vertices in $Q_M$ and, as two consecutive vertices cannot be marked, at least also $\rho$ regular vertices.

Now, let $\mathcal{B}$ be any possible choice of a subset of  $b$ of the $\rho$ marked vertices in $Q_M$. We call these marked vertices \emph{pre-boundary}, and  the $\rho-b=n$ remaining ones \emph{pre-cuspidal}, as the role that these vertices will play in what follows will be different. Label counterclockwise  $V_1, \ldots, V_M$ the vertices of $Q_M$ and assume that $V_1$ is a regular vertex. Define
$$
s_J=\#\{V_L: 1\leq L< J \text{ and } V_L \text{ is a pre-boundary vertex of } Q_M\}
$$
\noindent and let $\eta:\{1,\ldots,M\}\longrightarrow\{1,\ldots,m\}$ be the injective map defined by $\eta(J)=J+s_J$.

There exists, up to isometry, a unique hyperbolic $m$-sided polygon $P_m$ satisfying all the conditions described in \autoref{theo2} and such that, if we label $v_1,\ldots,v_m$ the vertices of $P_m$ counterclockwise:

\begin{itemize}
    \item $v_{\eta(J)}$ is an ideal vertex of $P_m$ if $V_J$ is a pre-cuspidal vertex of $Q_M$, 
    \item $v_{\eta(J)}$ and $v_{\eta(J)+1}$ are boundary vertices of $P_m$ if $V_J$ is a pre-boundary vertex of $Q_M$,
    \item all the remaining $v_j$ are interior vertices of $P_m$ (recall the notation in \autoref{theo2}).
\end{itemize}

If we label $d_j$ the edge of $P_m$ with vertices $v_{j-1}$ and $v_{j}$ (modulo $m$) and $q_J$ the edge of $Q_M$ with vertices $V_{J-1}$ and $V_J$ (modulo $M$), the side-pairing pattern in $Q_M$ defines a natural side-pairing pattern in $P_m$: two edges $d_{\eta(J)}$ and $d_{\eta(L)}$ in $P_m$ will be identified if and only if $q_J$ and $q_L$ are identified in $Q_M$ by some element  $\alpha\in\Lambda$, and the corresponding side-pairing transformation of $P_m$ is set to match the orientation preserving/reversing character of $\alpha$.

By this process we have defined the pairwise identification of $M$ of the total $m$ edges of $P_m$, while the remaining $m-M=b$ ones are edges of type II, which must be self-identified by a hyperbolic reflection, according to \autoref{theo2}. If we denote $\Gamma_{\Lambda, \mathcal{B}}$ the group generated by the side-pairing identifications in $P_m$ together with these hyperbolic reflections, $\Gamma_{\Lambda, \mathcal{B}}$ is a NEC group with fundamental domain $P_m$ by the Poincaré's Polygon Theorem and $S_{\Lambda, \mathcal{B}}=\mathbb{D}/\Gamma_{\Lambda, \mathcal{B}}$ is an extremal surface with Euler characteristic $\chi$, $n$ cusps and $b$ boundary components. In fact, we have constructed an extremal disc configuration based on $P_m$ for the triplet $(\chi,n,b)$, whose image under the EBF-correspondence is the starting subgroup $\Lambda< \Delta^{\pm}(2,3,M)$, hence EBF is a surjective map.

On the other hand, it is clear that the distinct EBF-preimages of a given $\Lambda< \Delta^{\pm}(2,3,M)$ are determined by the choice of the set $\mathcal{B}$ of pre-boundary vertices, hence there are at most $\binom{\rho}{b}$ of such preimages. Nevertheless, some of them may determine equivalent extremal disc configurations due to symmetry issues (see Examples \ref{ex:choice1} and \ref{ex:choice2} below).

To finish the proof, simply note that if $n=0$ or $b=0$ there is no freedom in the choice of $\mathcal{B}$, and therefore the EBF-correspondence is one-to-one in these cases.
\end{proof}

\begin{remark}\label{n_extr_disc_conf}
The proof of \autoref{count} shows that the consideration of the EBF-correspondence gives not only bounds for the number of extremal disc configurations, but even a practical procedure to explicitly obtain all such configurations. For a given triplet $(\chi,n,b)$, compute first the list of all the conjugacy classes of subgroups of $\Delta^{\pm}(2,3,M)$ satisfying the desired conditions. Then, for each group $\Lambda$, produce the NEC groups $\Gamma_{\Lambda, \mathcal{B}}$, corresponding to every possible choice $\mathcal{B}$ of $b$ marked vertices of $Q_M$, which we declare as pre-boundary. Finally, check which of these configurations are equivalent in order to avoid repetitions.
\end{remark}

From now on, we will describe the topology of any surface of genus $g$ with $n$ cusps and $b$ boundary components by the tuple  $(\pm,g,n,b)$, where the first symbol is $+$ if the surface is orientable and $-$ if not. We call this $4$-tuple the \emph{topological type} of the surface. Notice that the Euler characteristic $\chi$ of the surface is totally defined by the topological type, as also are the values $m$, $M$ and $\rho$ in the statements of \autoref{theo2} and \autoref{count}.

\begin{example}\label{ex:choice1} 
Combinatorics shows that there are only two conjugacy classes of subgroups of index $16$ of $\Delta^{\pm}(2,3,8)$ that have two non-conjugate elliptic elements of order three, whose side-pairing generators are illustrated at the left-hand side of \autoref{fig:fig_8_noori_n1b1}. A solid line means that the corresponding side-pairing transformation preserves orientation, while a dashed line means that it reverses orientation.

\begin{figure}[!htbp]
\begin{tabular}{c|c}
   \begin{subfigure}{.3\textwidth}
    \centering
    \includegraphics[width=0.6\linewidth]{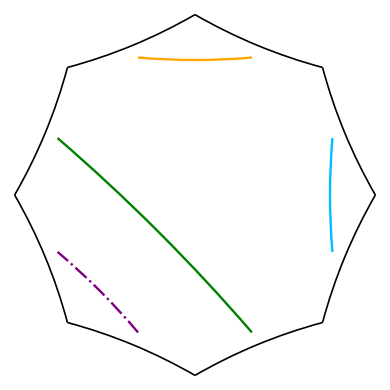}
    \vspace{0.8em}
    \end{subfigure}%
    & 
    \begin{subfigure}{.32\textwidth}
  \centering
  \includegraphics[width=0.7\linewidth]{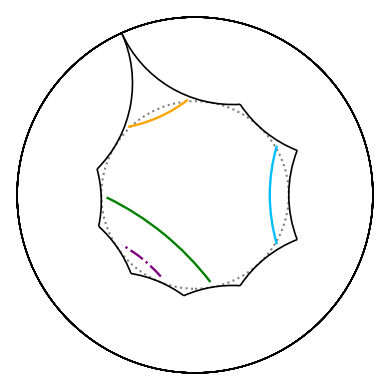}
\end{subfigure}%
\begin{subfigure}{.32\textwidth}
  \centering
  \includegraphics[width=0.7\linewidth]{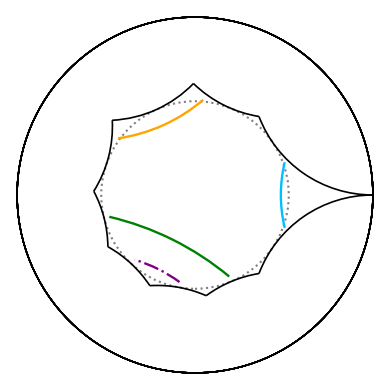}
\end{subfigure} \vspace{-1.1em} \\
\ & \\
 \hline \vspace{-0.8em} \\ 
    \begin{subfigure}{.3\textwidth}
    \centering
    \includegraphics[width=0.6\linewidth]{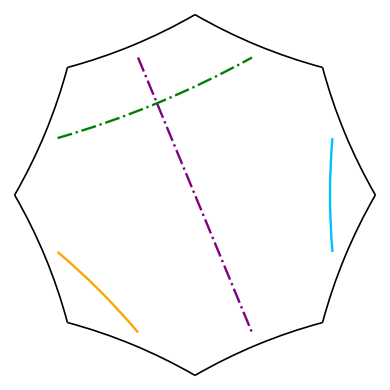}
    \vspace{0.4em}
    \end{subfigure}%
    &
    \begin{subfigure}{.32\textwidth}
  \centering
  \includegraphics[width=0.7\linewidth]{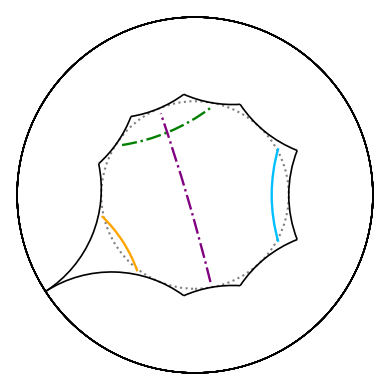}
  \vspace{-0.5em}
\end{subfigure}%
\begin{subfigure}{.32\textwidth}
  \centering
  \includegraphics[width=0.7\linewidth]{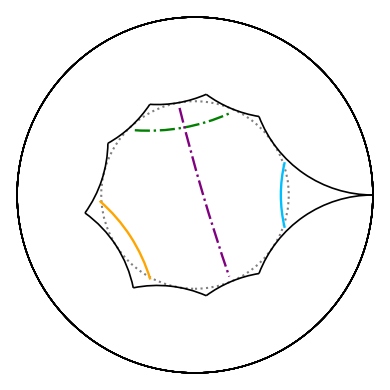}
  \vspace{-0.5em}
\end{subfigure}
\end{tabular}
\caption{In each of the rows: on the right-hand side, the two (equivalent) extremal disc configurations with topological type $(-,1,1,1)$ obtained as the EBF-preimages of the subgroup of $\Delta^{\pm}(2,3,8)$ shown on the left.}
\label{fig:fig_8_noori_n1b1}
\end{figure}

According to \autoref{n_extr_disc_conf}, the extremal disc configurations for $\chi=-1$, $n=1$, $b=1$ can be obtained as the preimages of these subgroups by the EBF-correspondence, after choosing one of the elliptic fixed point vertices to play the role of a pre-boundary marked vertex. In both cases, we can see that the two possible elections of vertex will determine equivalent extremal disc configurations (see the right-hand side of \autoref{fig:fig_8_noori_n1b1}).

Thus, the number of extremal disc configurations $N_{-1,1,1}$, that in this case  corresponds only to surfaces of topological type $(-,1,1,1)$, agrees with $\# C_{8,2}=2$, even  though $n\neq 0$ and $b\neq 0$.
\end{example}

\begin{example}\label{ex:choice2}
There are four conjugacy classes of subgroups of index $24$ of $\Delta^\pm (2,3,12)$ with three elliptic elements of order $3$. The extremal disc configurations for $\chi=-2, n=2, b=1$
can be obtained as the EBF-preimages of these four groups after choosing which of the three elliptic fixed points will be a pre-boundary vertex (see \autoref{fig:fig_12_noori_n2b1}). According to \autoref{count}, we know that $4\le N_{-2,2,1}\le 12$.

\begin{figure}[!htbp]
\begin{tabular}{c|c}
   \begin{subfigure}{.21\textwidth}
    \centering
    \includegraphics[width=0.85\linewidth]{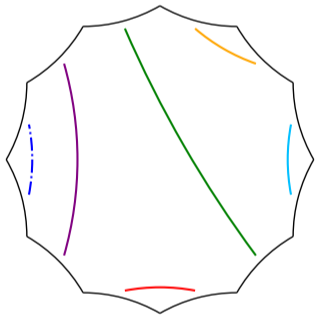}
    \vspace{0.8em}
    \end{subfigure}%
    & 
    \begin{subfigure}{.23\textwidth}
  \centering
  \includegraphics[width=0.97\linewidth]{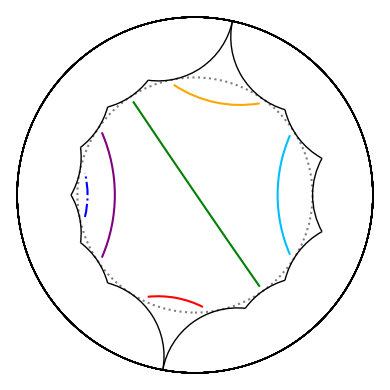}
\end{subfigure}%
\begin{subfigure}{.23\textwidth}
  \centering
  \includegraphics[width=0.97\linewidth]{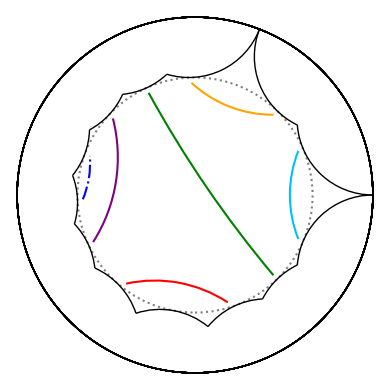}
\end{subfigure}
\begin{subfigure}{.23\textwidth}
  \centering
  \includegraphics[width=0.97\linewidth]{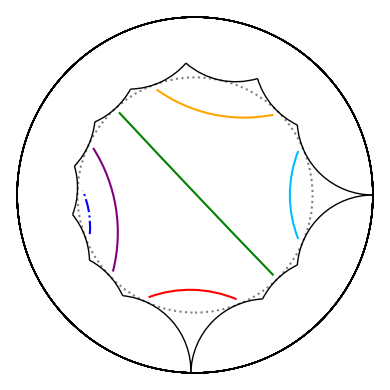}
\end{subfigure} \vspace{-1.1em} \\
\ & \\
 \hline \vspace{-0.8em} \\ 
    \begin{subfigure}{.21\textwidth}
    \centering
    \includegraphics[width=0.85\linewidth]{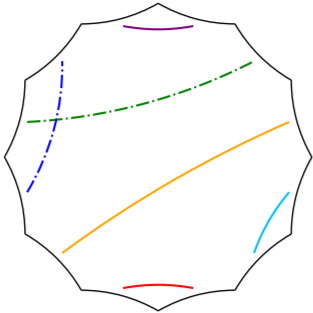}
    \vspace{0.8em}
    \end{subfigure}%
    & 
    \begin{subfigure}{.23\textwidth}
  \centering
  \includegraphics[width=0.97\linewidth]{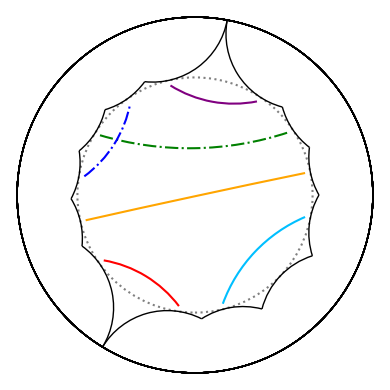}
\end{subfigure}%
\begin{subfigure}{.23\textwidth}
  \centering
  \includegraphics[width=0.97\linewidth]{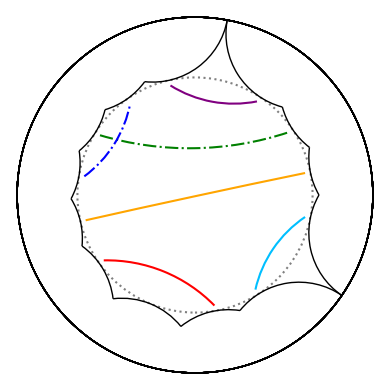}
\end{subfigure}
\begin{subfigure}{.23\textwidth}
  \centering
  \includegraphics[width=0.97\linewidth]{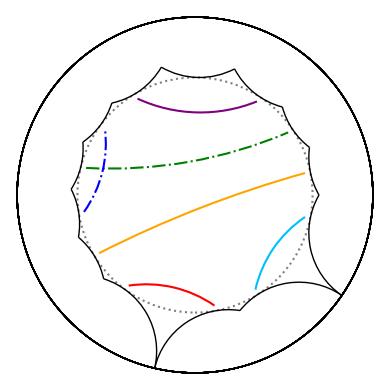}
\end{subfigure} \vspace{-1.1em} \\
\ & \\
 \hline \vspace{-0.8em} \\ 
 \begin{subfigure}{.21\textwidth}
    \centering
    \includegraphics[width=0.85\linewidth]{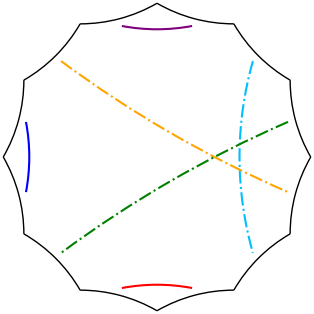}
    \vspace{0.8em}
    \end{subfigure}%
    & 
    \begin{subfigure}{.23\textwidth}
  \centering
  \includegraphics[width=0.97\linewidth]{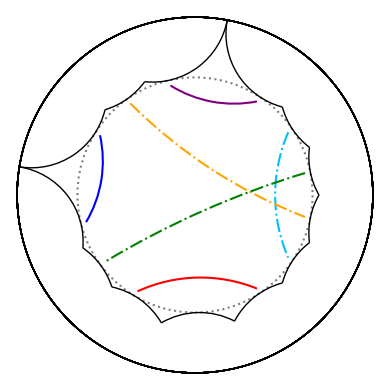}
\end{subfigure}%
\begin{subfigure}{.23\textwidth}
  \centering
  \includegraphics[width=0.97\linewidth]{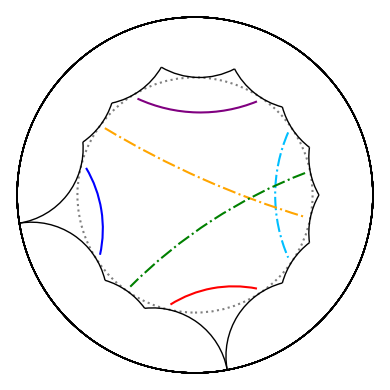}
\end{subfigure}
\begin{subfigure}{.23\textwidth}
  \centering
  \includegraphics[width=0.97\linewidth]{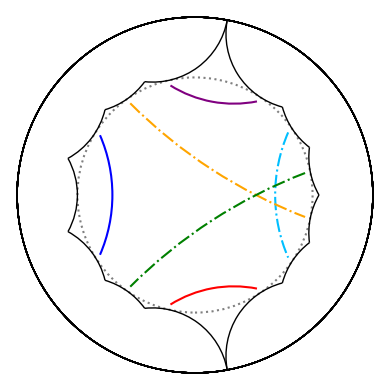}
\end{subfigure} \vspace{-1.1em} \\
\ & \\
 \hline \vspace{-0.8em} \\ 
 \begin{subfigure}{.21\textwidth}
    \centering
    \includegraphics[width=0.85\linewidth]{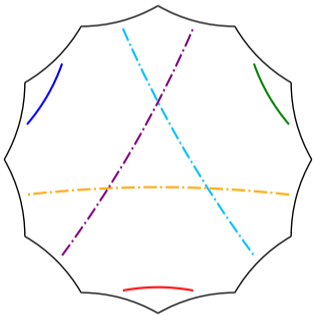}
    \vspace{0.8em}
    \end{subfigure}%
    & 
    \begin{subfigure}{.23\textwidth}
  \centering
  \includegraphics[width=0.97\linewidth]{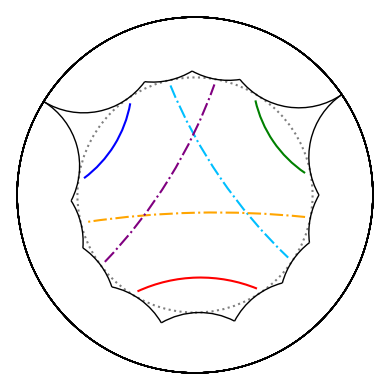}
\end{subfigure}%
\begin{subfigure}{.23\textwidth}
  \centering
  \includegraphics[width=0.97\linewidth]{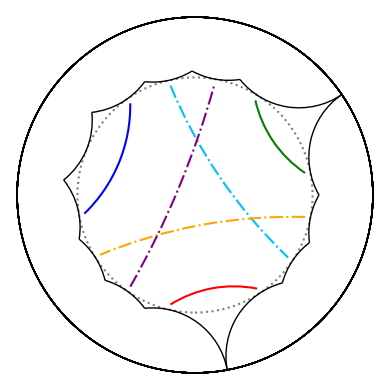}
\end{subfigure}
\begin{subfigure}{.23\textwidth}
  \centering
  \includegraphics[width=0.97\linewidth]{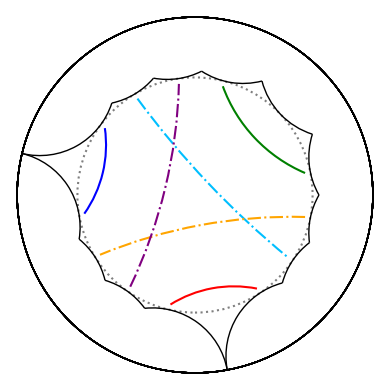}
\end{subfigure}
\end{tabular}
\caption{On the left: subgroups of $\Delta^{\pm}(2,3,12)$ used to construct, via the EBF-correspondence, the extremal disc configurations for $\chi=-2$, $n=2$, $b=1$. On the right, the EBF-preimages: the first two configurations in the third row and all three in the fourth row turn out to be equivalent.}
\label{fig:fig_12_noori_n2b1}
\end{figure}

We see that this election determines three different extremal disc configurations for the first two cases, only two for the third case and just a single one for the last group, hence $N_{-2,2,1}=9$. All these configurations correspond to the topological type $(-,1,2,1)$.
\end{example}

 Examples \ref{ex:choice1} and \ref{ex:choice2} show how the existence of automorphisms of the hyperbolic orbifold described by a given subgroup $\Lambda < \Delta^\pm(2,3,M)$ in the conditions of \autoref{count} affects the size of the preimage of $\Lambda$ under the EBF-correspondence. In fact, one can easily check that the following general statement is true:

\begin{proposition} \label{pr:pre-EBF}
    Let $\Lambda$ be a subgroup of index $2M$ of $\Delta^{\pm}(2,3,M)$ with $\rho$ elliptic elements of order $3$ and let $b\in\{0,\ldots,\rho\}$. Denote $\mathcal{M}(\Lambda)$ the set of $\rho$ elliptic fixed point vertices in the obvious regular $M$-gon $Q_M$ of angle $2\pi/3$ that serves as fundamental domain for $\Lambda$ and $\mathcal{B}^b_{\mathcal{M}(\Lambda)}$ the collection of subsets of $b$ elements of $\mathcal{M}(\Lambda)$. Then
 $$|\text{EBF}^{-1}([\Lambda])|=\#\{\text{orbits of the  action } N^{\pm}(\Lambda)/\Lambda\times \mathcal{B}^b_{\mathcal{M}(\Lambda)}\longrightarrow \mathcal{B}^b_{\mathcal{M}(\Lambda)}\}$$
 where $N^{\pm}(\Lambda)$ is the normalizer of $\Lambda$ in $\Delta^{\pm}(2,3,M)$.
 \end{proposition}

Since two extremal disc configurations can only be equivalent if they are mapped to the same element under the EBF-correspondence, the next result follows from \autoref{pr:pre-EBF}:

\begin{corollary}\label{count_edc}
    Let $\chi<0$ and $n,b\geq 0$ be integer numbers such that $M=6-6\chi-2\rho>6$, for $\rho=n+b$. Then
$$   N_{\chi,n,b} =\sum_{[\Lambda] \in C_{M,\rho}} \#\{\text{orbits of the action  } N^\pm (\Lambda)/\Lambda\times \mathcal{B}^b_{\mathcal{M}(\Lambda)}\longrightarrow \mathcal{B}^b_{\mathcal{M}(\Lambda)}\}.$$
\end{corollary}

This is an effective formula for the computation of $N_{\chi, n,b}$. For instance, for small values of the Euler characteristic, we have the following:

\begin{theorem} \label{th:tables}
The exact number of extremal disc configurations whose underlying surface has Euler characteristic $\chi=-1$ or $-2$ is displayed in Table \ref{table:euler-1,-2},  classified according to the  topological types $(\pm, g, n, b)$ involved.

\emph{
\begin{table}[!htbp]
    \begin{center}
    \begin{tabular}{| c | c | c | c | c | c |}
    \hline
    \multicolumn{2}{ |c| }{Triplet} & \multicolumn{2}{ |c| }{Orientable } & \multicolumn{2}{ |c| }{Non-orientable } \\ \hline
    \multicolumn{1}{ |c| }{ $(\chi,n,b)$ } & \multicolumn{1}{ |c| }{$N_{\chi,n,b}$} &  \multicolumn{1}{ |c| }{Top. type } & \multicolumn{1}{ |c| }{Configurations }
    & \multicolumn{1}{ |c| }{Top. type } & \multicolumn{1}{ |c| }{Configurations }
    \\ \hline
    (-1,0,0) & 11 & - & - & (-,3,0,0) & 11 \\ \hline
     (-1,1,0) & 5  & (+,1,1,0) & 1 & (-,2,1,0) & 4  \\ \hline
     (-1,0,1) & 5  & (+1,0,1) & 1  & (-,2, 0, 1) & 4 \\ \hline
     (-1,1,1) & 2  & - & - & (-,1,1,1) & 2 \\ \hline
     (-1,0,2) & 2  & - & - & (-,1,0,2) & 2 \\ \hline
     (-1,2,0) & 2  & - & -  & (-,1,2, 0) & 2 \\ \hline
     (-1,3,0) & 1  & (+,0,3,0) & 1 & -  & - \\ \hline
     (-1,0,3) & 1  & (+,0,0,3) & 1 & - & - \\ \hline
     (-1,1,2) & 1  & (+,0,1,2) & 1 & - & - \\ \hline
     (-1,2,1) & 1  & (+,0,2,1) & 1 & - & - \\
     \hline 
        (-2,0,0) & 152  & (+,2,0,0) & 8 & (-,4,0,0) & 144 \\ \hline
     (-2,1,0) & 69  & - & - & (-,3,1,0) & 69  \\ \hline
     (-2,0,1) & 69  & - & -  & (-,3,1, 0) & 69  \\ \hline
     (-2,1,1) & 35  & (+,1,1,1) & 5 & (-,2,1,1) & 30 \\ \hline
     (-2,0,2) & 27  & (+,1,0,2) & 5 & (-,2,0,2) & 22 \\ \hline
     (-2,2,0) & 27  & (+,1,2,0) & 5  & (-,2,2, 0) & 22 \\ \hline
     (-2,3,0) & 4  & - & - & (-,1,3,0)  & 4 \\ \hline
     (-2,0,3) & 4  & - & - & (-,1,0,3) & 4 \\ \hline
     (-2,1,2) & 9  & - & - & (-,1,2,1) & 9 \\ \hline
     (-2,2,1) & 9  & - & - & (-,1,1,2) & 9 \\ \hline
     (-2,4,0) & 1  & (+,0,4,0) & 1 & - & - \\
     \hline
     (-2,0,4) & 1  & (+,0,0,4) & 1 & - & - \\
     \hline
     (-2,3,1) & 1  & (+,0,3,1) & 1 & - & - \\
     \hline
     (-2,1,3) & 1 & (+,0,1,3) & 1 & - & - \\
     \hline
     (-2,2,2) & 3  & (+,0,2,2) & 3 & - & - \\
     \hline
    \end{tabular}
    \end{center}
    \caption{Number of extremal disc configurations of underlying hyperbolic surfaces with Euler characteristic $\chi=-1$ and $\chi=-2$, classified by their topological type.}\label{table:euler-1,-2}
    \end{table}
}
\end{theorem}
 
As we have mentioned in \autoref{n_extr_disc_conf}, 
the first step towards a full description of all the possible extremal disc configurations for a given topological data, using our EBF-method, is the complete determination of all the relevant subgroups of the corresponding triangle group. This can be done by direct inspection in particular cases like \autoref{ex:choice1}, but it generally requires the use of software like \cite{GAP4}. The second step is the application of \autoref{count_edc}, using again GAP for the computation of the normalizers of these groups.

Many of the extremal disc configurations that are referred to in \autoref{th:tables} are illustrated in sections \ref{section:location} and \ref{sec:counting}. The interested reader can check the GAP programs that we have used to obtain this information in \cite{girondo-munozGitHub}, as well as the outputs produced by these programs.

\begin{remark} \label{rem:orient_withoutboundary}
In the particular case of orientable hyperbolic surfaces without boundary, it would be natural 
to consider two fundamental domains in the conditions of \autoref{theo2} to define the same extremal disc configuration only if the isometry $\tau$ in Definition \ref{def:extrdiscconf} preserves the orientation, as otherwise the underlying Riemann surfaces could not be isomorphic. Let us call the conjugacy classes of this new equivalence relation \emph{orientable extremal disc configurations}. All  side-pairing identifications of the fundamental domains in this particular situation are orientation-preserving, so we can restrict ourselves to the Fuchsian setting, and by the same argument as in the proof of \autoref{count} we have the following correspondence:

\medskip

\noindent
\begin{minipage}{0.46\textwidth}
\fbox{%
\parbox{0.95\textwidth}{
 Orientable extremal disc configurations of orientable surfaces of genus $g$, $n$ cusps and without boundary.
  }%
}
\end{minipage}
\begin{minipage}{0.075\textwidth}
 $\xrightarrow{EBF^*}$
\end{minipage}
\begin{minipage}{0.46\textwidth}
\fbox{%
\parbox{0.95\textwidth}{
 Conjugacy classes of subgroups of index $M$ of $\Delta(2,3,M)$ with $n$ non-conjugate elliptic elements, all of them of order 3, where $M=12+4n-6$.
  }%
  }
\end{minipage}

\medskip

This $\text{EBF}^*$-correspondence leads to six different orientable extremal disc configurations when $g=1$, $n=2$ while, according to \autoref{table:euler-1,-2}, there are five extremal disc configurations for the topological type $(+,1,2,0)$. The point is that two of the six different orientable extremal disc configurations, those depicted in \autoref{fig_remarkori}, are identified by an orientation-reversing isometry, so they define the same extremal disc configuration. However, in principle, it is uncertain whether or not the corresponding underlying Riemann surfaces are isomorphic. We come back to this question in \autoref{ex:twicepunct_torus}. 
 
\begin{figure}[!htbp]
\begin{subfigure}{.33\textwidth}
  \centering
  \includegraphics[width=.76\linewidth]{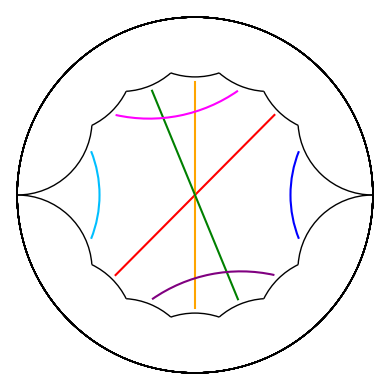}
  \label{fig5:sfig1}
\end{subfigure}%
\begin{subfigure}{.33\textwidth}
  \centering
  \includegraphics[width=.76\linewidth]{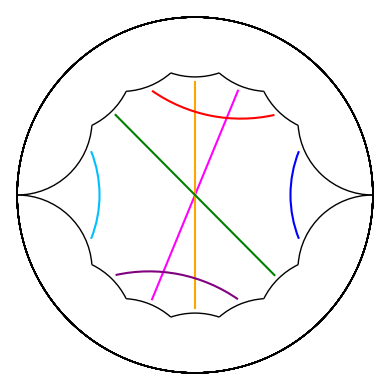}
  \label{fig5:sfig2}
\end{subfigure}
\caption{Two orientable extremal disc configurations of topological type $(+,1,2,0)$.}
\label{fig_remarkori}
\end{figure}

\end{remark}

\section{A second application: locating hidden  extremal discs centers}
\label{section:location}

We have already mentioned that the study of extremal surfaces was developed first for closed surfaces. In this context, the characterization of the groups that uniformize extremal surfaces was essential, for instance, for showing that a compact orientable surface of genus greater than 3 can have at most one extremal disc (\cite{Girondo_Gonzalez-Diez_1999}), and that the  same statement holds true for non-orientable surfaces of genus greater than 6 (see \cite{Girondo_Nakamura_2007}).
The explicit knowledge of the uniformizing groups has been also the main tool for the complete study of the location of every extremal disc within a given hyperbolic surface which, as a byproduct, allows the full 
 determination of the automorphism group of extremal surfaces.

The way this research has been carried out in the past, for a closed extremal surface $S$, is based in the following fact: if $S$ is uniformized as $S\simeq \mathbb{D}/\Gamma$ and $[z]_{\Gamma}$ is the center of an extremal disc, the $\Gamma$-images of the Dirichlet domain of $\Gamma$ centered at $z$ produce a tessellation of $\mathbb{D}$ by isometric regular polygons of angle $2\pi /3$. The set $\mathcal{D}$ of distances between different centers of polygons of such a tessellation, which depends solely on the topology of $S$, plays a fundamental role in the location of ``hidden'' extremal disc centers in $S$, as a necessary condition for $[w]_{\Gamma}\in S$ to be such a center is that $d_\mathbb{D}\big(w, \gamma(w)\big)\in \mathcal{D}$ for every $\gamma \in \Gamma$.

 The presence of cusps and/or boundary components impose additional metric restictions on the location of extremal discs. Particularly, the injectivity radius of a hyperbolic surface $S$ at the point $q\in S$ is given by

 \vspace{-0.5em}
 
 $$
\text{injrad}_q(S)=\min\left\{\frac{1}{2} \text{sys}_q(S), \text{dist}(q,\partial S)\right\}
$$
\noindent
where we set $\text{dist}(q,\partial S)=\infty$ if $S$ has no boundary and $\text{sys}_q(S)$ stands for the \emph{systole} of $S$ at $q$ (i.e. the length of the shortest closed geodesic in $S$ based at the point $q$). Note that when $q$ approaches a cusp of $S$ then $\text{sys}_q(S)\rightarrow 0$,  preventing the existence of extremal discs ``close'' to cusps. Notice also that there cannot be an extremal disc near a boundary component as the center $q\in S$ of an extremal disc must satisfy $\text{dist}(q,\partial S)\geq r_{\chi,n,b}$.
 
On the other hand, finding the location of extremal discs within extremal surfaces with boundary and/or cusps gets more complicated as the tessellation determined by an extremal disc, hence the set $\mathcal{D}$ of distances, can differ from one disc configuration to another.

Fix integer numbers $\chi<0$ and $n,b\ge 0$, and let $\mathcal{C}_j$ be the $j$-th of the corresponding $N_{\chi,n,b}$ extremal disc configurations. Assume that $\mathcal{C}_j$ is determined by the Dirichlet polygon $P_j=D_{\Gamma_j}(0)$ for a NEC group $\Gamma_j$, together with the corresponding side-pairing pattern.

 A fundamental difference with respect to the case of closed extremal surfaces is that two such polygons $P_j$ and $P_k$ may not be isometric to each other. In any case, the configuration $\mathcal{C}_j$ still determines a set of \emph{admissible distances} 
 $$
\mathcal{D}_j=\left\{ d_\mathbb{D}\big(0, \gamma(0)\big) : \gamma \in \Gamma_j \right\}\setminus\{0\}
$$

\noindent but now the sets $\mathcal{D}_{j}$ and $\mathcal{D}_{k}$ may not be the same when $j \neq k$.

The extremal surface $S_j = \mathbb{D}/\Gamma_j$ contains, by definition, an extremal disc that is centered at $[0]_{\Gamma_j}\in S_j$, but it may contain other \emph{hidden} extremal discs centered elsewhere. Observe that a necessary condition for the existence of such a disc is the following:

\begin{lemma}\label{lemma:admissible}
If $[w]_{\Gamma_j}$ is the center of an extremal disc in  $S_j $, there exists $k\in\{1,\ldots, N_{\chi,n,b}\}$ such that 
$d_\mathbb{D}\big(w, \gamma(w)\big)\in \mathcal{D}_k$ for all $\gamma \in \Gamma_j$.
\end{lemma}

It might occur that $k=j$ does not work and, in this case, the extremal disc configuration that corresponds to the  disc centered at $[w]_\Gamma$ is necessarily inequivalent to $\mathcal{C}_j$. 
It might also happen that $k=j$ works but $D_{\Gamma_j}(0)$ and $D_{\Gamma_j}(w)$ differ as extremal disc configurations. On the other hand, if $D_{\Gamma_j}(w)$ and $D_{\Gamma_j}(0)$ are equivalent as extremal disc configurations, there exists a (conformal or anticonformal) automorphism of the underlying surface $S_j$ that maps $[0]_{\Gamma_j}$ to $[w]_{\Gamma_j}$ (and, accordingly, one extremal disc into the other).

 \autoref{lemma:admissible} suggests that the location of the preimages of centers of hidden extemal discs on any $\mathcal{C}_j$ are affected by metric conditions that depend on the whole list of uniformizing groups $\{\Gamma_k:1\leq k\leq N_{\chi,n,b}\}$. However, \autoref{admissible_distances} below  shows that, even though the set of admissible distances $\mathcal{D}_k$ depends on the particular extremal disc configuration $\mathcal{C}_k$, the first two elements in  $\mathcal{D}_k$ are the same for all $k$, simplifying the study in many of the cases.

 Recall that, according to \autoref{theo2}, if $S=\mathbb{D}/\Gamma$ has an extremal disc centered at $[0]_\Gamma$,
 $D_{\chi,n,b}$ is the distance from the origin to any boundary vertex of $D_\Gamma(0)$, $d_{\chi,n,b}$ is the distance from the origin to a interior vertex of $D_\Gamma(0)$ and $l_{\chi,n,b}$ is the length of the edges of type IV (see \autoref{fig_distances}).

 \begin{figure}[!hbtp]
\begin{subfigure}{.33\textwidth}
  \centering
  \includegraphics[width=.78\linewidth, draft=False]{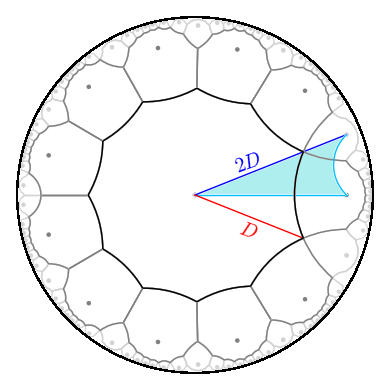}
  \label{fig_dist:sfig1}
\end{subfigure}%
\begin{subfigure}{.33\textwidth}
  \centering
  \includegraphics[width=.78\linewidth, draft=False]{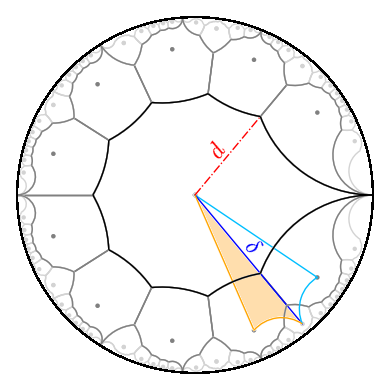}
  \label{fig_dist:sfig2}
\end{subfigure}
\caption{The second admissible distance for the cases $b> 0$ and $b=0$ respectively.}
\label{fig_distances}
\end{figure}

\begin{lemma}\label{admissible_distances}
    Given integers $\chi<0$ and $n,b\geq 0$, the first admissible distance associated to any extremal disc configuration of a hyperbolic surface with Euler characteristic $\chi$, $n$ cusps and $b$ boundary components is $2 r_{\chi,n,b}$. 
    
    If $b>0$, the second admissible distance is $2D_{\chi,n,b}$ while if $b=0$ and $(\chi,n,b)\neq (-1,3,0)$, the second admissible distance is $\delta_{\chi,n,b}=2d_{\chi,n,b}+l_{\chi,n,b}$.
\end{lemma}

\begin{proof}
Given an extremal disc configuration $\mathcal{C}$ for the triplet $(\chi,n,b)$, there exists a NEC group $\Gamma$ such that $S=\mathbb{D}/\Gamma$ defines an extremal surface with Euler characteristic $\chi$, $n$ cusps and $b$ boundary components that has an extremal disc centered at $[0]_\Gamma$, and such that the Dirichlet domain $D_\Gamma(0)$ and the side-pairing  pattern defined on $D_\Gamma(0)$ by $\Gamma$  correspond to $\mathcal{C}$. Clearly, the shortest admissible distance must be $d_\mathbb{D}\big(0,\gamma(0)\big)=2r_{\chi,n,b}$, which is attained when $\gamma\in\Gamma$ is any side-pairing identification of $D_\Gamma(0)$.

If $d^*>2r_{\chi,n,b}$ is the second admissible distance,  
$d\big(0,\gamma(0)\big)=d^*$ for some $\gamma\in\Gamma$ given as the composition of two side-pairing transformations of $D_\Gamma(0)$, namely $\gamma=s_1\circ s_0$. Consider the hyperbolic triangle with vertices at the points $0$, $s_0(0)$ and $\gamma(0)$ and denote $\theta\in(0,\pi)$ the angle at $s_0(0)$. By the hyperbolic law of cosines
\begin{equation*}
    \cosh(d^*)=\cosh^2(2r_{\chi,n,b})-\sinh^2(2r_{\chi,n,b})\cos(\theta).
\end{equation*}

If $b>0$, consider the triangle with vertices $0$, $\eta_1(0)$ and $\eta_2\circ\eta_1(0)$, where $\eta_1\in\Gamma$ is the hyperbolic reflection in some edge of type II and $\eta_2\in\Gamma$ is the hyperbolic isometry that identifies its two adjacent edges of type III (see the left-hand side of \autoref{fig_distances}).
By the hyperbolic law of cosines
\begin{equation*}
\cosh(2D_{\chi,n,b})=\cosh^2(2r_{\chi,n,b})-\sinh^2(2r_{\chi,n,b})\cos(\gamma).
\end{equation*}

Consequently, $d^*<2D_{\chi,n,b}$ if and only if $\theta<\gamma$.
On the other hand, $\theta$ can be expressed as a sum of some subangles, each of them equal either to $\alpha, 2\beta$ or $\gamma$ by the characterization in \autoref{theo2} (see \autoref{fig:angles}). By the inequalities exhibited in \autoref{angles}, the only possibility leading to $d^*<2D_{\chi,n,b}$ would be  $\theta=\alpha$.
However, $\alpha$ is the angle of an equilateral triangle with sidelength $2r_{\chi,n,b}$ and $\theta=\alpha$ implies
\begin{equation*}
    \cosh(2r_{\chi,n,b})=\cosh^2(2r_{\chi,n,b})-\sinh^2(2r_{\chi,n,b})\cos(\alpha)=\cosh(d^*),  
\end{equation*}
hence $d^*=2r_{\chi,n,b}$, which is a contradiction. Therefore, the second admissible distance $d^*$ corresponds to the case $\theta=\gamma$ if $b> 0$ and it is precisely $2D_{\chi,n,b}$.

\begin{figure}[!htbp]
    \centering
    \includegraphics[width=0.29\linewidth, draft=False]{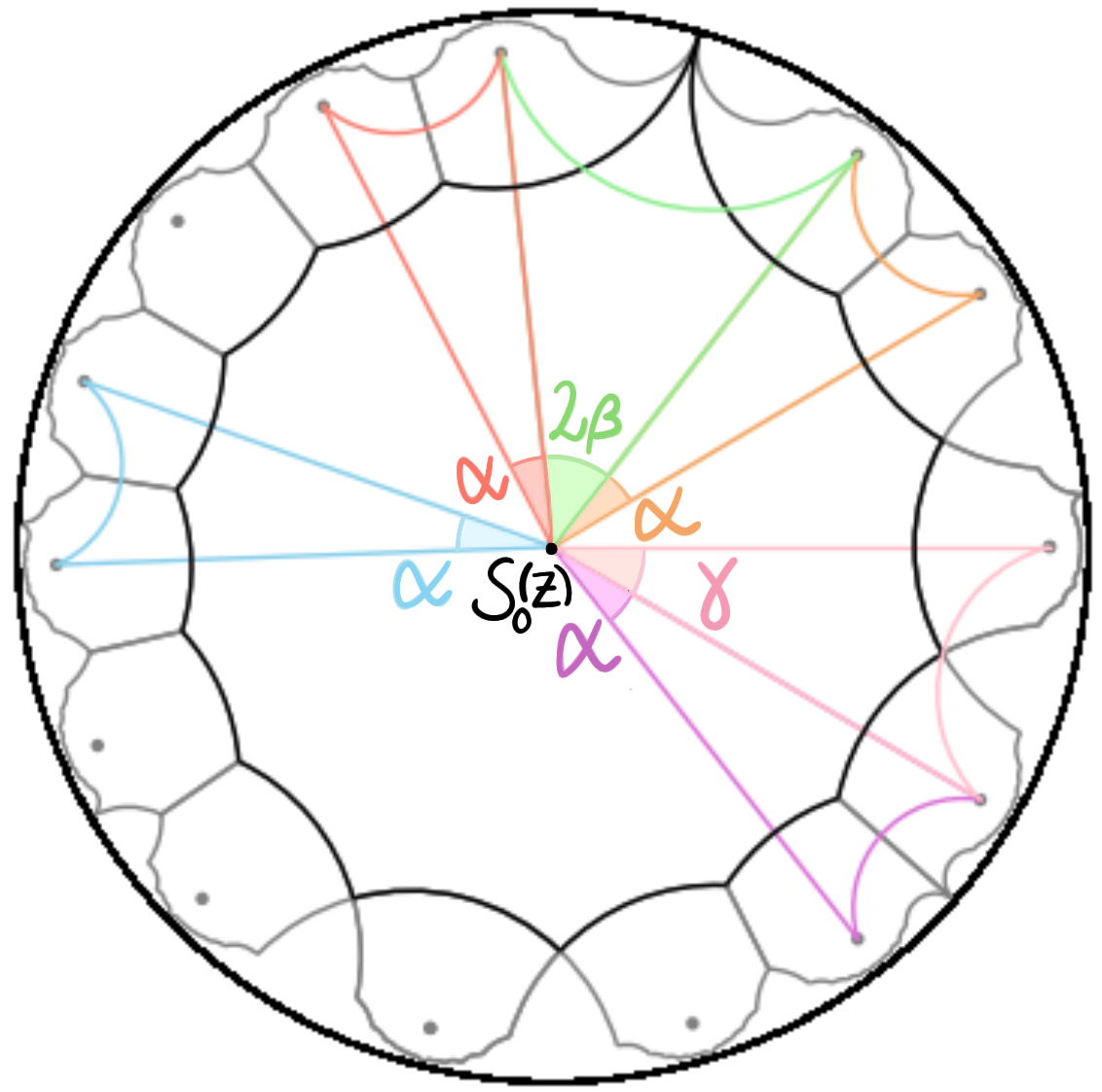}
    \caption{The angle $\theta$ in $s_0(0)$ must be one of the angles illustrated here (or maybe the sum of some of them).}
    \label{fig:angles}
\end{figure}

In the case  $b=0$, if $(\chi,n,b)\neq (-1,3,0)$ we can reproduce this argument, obtaining that the second admissible distance is now  $\delta_{\chi,n,b}=2d_{\chi,n,b}+l_{\chi,n,b}$. The key point here is the existence of an edge of type IV, something guaranteed by \autoref{ex_sph_three}.
\end{proof}

 There is a natural way of rephrasing the statement of \autoref{lemma:admissible} that helps us to develop a better intuition of what points of the fundamental domain described by an extremal disc configuration are \emph{candidates} to be the preimage of the center of a hidden extremal disc of the corresponding extremal surface. The key point of this insight is the usage of ED-cycles (short for \emph{equi-displaced cycles}):

\begin{definition}
    Given an isometry $\gamma$ in $\mathbb{D}$, the \emph{ED-cycle} of $\gamma$ associated to a distance $d>0$ is the set of points in $\mathbb{D}$ that are moved a distance $d$ by $\gamma$, that is, the set
    $$b_\gamma(d)=\{z\in\mathbb{D}:d_\mathbb{D}\big(z,\gamma(z)\big)=d\}.$$
\end{definition}

The set $b_\gamma(d)$ can be re-interpreted as the preimage of $d>0$ by the displacement function $z\longmapsto d_\mathbb{D}\big(z,\gamma(z)\big)$. This function was already studied in \cite{beardon} and, in fact, the following lemma, which is just a rewriting of [\cite{beardon}, Theorem 7.35.1] describes the particular shape that these sets have:

\begin{lemma}\label{move_isometries}
Given an isometry $\gamma$, the value $d_\mathbb{D}\big(z,\gamma (z)\big)$ satisfies the following identities:
\begin{enumerate}
    \item If $\gamma$ is a hyperbolic isometry with axis $A$ and translation length $T$, then 
    \begin{equation*}
        \sinh\big(d_\mathbb{D}\big(z,\gamma( z)\big)/2\big)=\cosh\big(d_\mathbb{D}(z,A)\big)\sinh(T/2).
    \end{equation*}
    \item If $\gamma$ is a glide reflection with axis $A$ and translation length $T$, then
    \begin{equation*}
        \cosh\big(d_\mathbb{D}\big(z,\gamma (z)\big)/2\big)=\cosh\big(d_\mathbb{D}(z,A)\big)\cosh(T/2).
    \end{equation*}
    \item If $\gamma$ is a parabolic isometry with ideal fixed point $u$ then $P(z,u)\sinh\left(d_\mathbb{D}\big(z,\gamma (z)\big)/2\right)$ is constant, where $P(z,u)$ is  the Poisson kernel.
\end{enumerate}
\end{lemma}

Fix some integers $\chi<0$ and $n,b\geq 0$ and take an extremal disc configuration $\mathcal{C}$ for the triplet $(\chi,n,b)$. Then $\mathcal{C}$ is given by a Dirichlet domain $D_\Gamma(0)$ and the side-pairing pattern defined by the uniformizing group $\Gamma$ of an extremal surface $S\simeq \mathbb{D}/\Gamma$ with an extremal disc centered at $[0]_\Gamma$.
Let  $\mathcal{D}_{\chi,n,b}$ be the ordered list obtained by putting together all the lists of admissible distances of the extremal disc configurations associated to that triplet. For any given $\gamma \in \Gamma$,  we say that the ED-cycle $b_\gamma(d)$ is \emph{admissible} if  $d\in\mathcal{D}_{\chi,n,b}$.

If there exists another extremal disc in $S$ centered at $[w]_\Gamma$, then $d_\mathbb{D}\big(w,\gamma(w)\big)\in \mathcal{D}_{\chi,n,b}$ for every side-pairing identification $\gamma$ of $D_\Gamma(0)$ by \autoref{lemma:admissible}. This means that $w$ must lie in the intersection of $3-3\chi-n$ admissible ED-cycles, one for each element of the set $\mathcal{S}$ of side-pairing identifications given by $\Gamma$. By the characterization provided in \autoref{theo2}, any  admissible ED-cycle always belong to one of the cases in \autoref{move_isometries}, unless it is associated to a hyperbolic reflection. In particular, there are only two options:

\begin{itemize}
    \item If $\gamma\in\mathcal{S}$ is not a parabolic isometry, $b_\gamma(d)$ defines generally a couple of equidistant hypercycles from the axis $A_\gamma$ of $\gamma$. In particular, if $\gamma$ is a hyperbolic reflection, then $b_\gamma(d)$ are the two hypercycles at distance $d/2$ from $A_\gamma$. If $\gamma$ is a hyperbolic isometry or a glide reflection, \autoref{move_isometries} provides the distance from both hypercycles to $A_\gamma$ when $d>T$ (when $d=T$, the ED-cycle $b_\gamma(d)$ is just one curve, the axis $A_\gamma$, and it is a geodesic).
    \item If $\gamma\in\mathcal{S}$ is a parabolic isometry, $b_\gamma(d)$ is simply a horocycle centered at the ideal fixed point of $\gamma$.
\end{itemize}

Denote by $B_\gamma(d)$ the open region bounded by the ED-cycle $b_\gamma(d)$, which can be either a hypercyclic region or a horoball. Observe that $B_\gamma(d_1)\subset B_\gamma(d_2)$ if $d_1<d_2$, and define $B_\gamma(d_1,d_2)=B_\gamma(d_2)\setminus \overline {B_\gamma(d_1)}$. Then, if $d_1,d_2 \in \mathcal{D}_{\chi,n,b}$ are two consecutive admissible distances, no admissible ED-cycle associated to $\gamma$ meets $B_\gamma(d_1,d_2)$ and, consequently, there cannot be any candidate in this region. The same statement holds for $B_\gamma(2r_{\chi,n,b})$ according to \autoref{admissible_distances}. For this reason, we will refer to the domains $B_\gamma(2r_{\chi,n,b})$ and $B_\gamma(d_1,d_2)$ where $d_1$ and $d_2$ are two consecutive admissible distances as \emph{forbidden regions}.

This setting encodes the idea that the center of extremal discs must lie away from the cusps of the surface: if $w\in D_\Gamma(0)$ is a representative of such a disc center, then $w$ must be outside the open horoball $B_{\gamma_u}(2r_{\chi,n,b})$, where $\gamma_u\in\Gamma$ is any parabolic side-pairing of $D_\Gamma(0)$ and $u$ is its ideal fixed point. Notice that the origin itself is certainly in the horocycle 
\begin{equation}\label{first_banana_parabolic}
    b_{\gamma_u}(2r_{\chi,n,b})=\{z\in\mathbb{D}:|z-u/2|=1/2\},
\end{equation}
\noindent so the candidates in $D_\Gamma(0)$ are at a finite distance $d_0$ from the origin, as we have ruled out a whole region around each cusp $[u]_\Gamma$. In fact, the following lemma shows that $d_0=d_{\chi,n,b}$.

\begin{lemma}\label{lemma:distance_origin}
Let $S\simeq \mathbb{D}/\Gamma$ be a hyperbolic surface with an extremal disc centered at $[0]_\Gamma$. If $w\in D_\Gamma(0)$ is a lift of any other extremal disc center in $S$, then $d_\mathbb{D}(0,w)\leq d_{\chi,n,b}$.
\end{lemma}

\begin{proof}
    Consider the decomposition of the  domain $D_\Gamma(0)$ into pieces described in \autoref{fig:decomposition}. If $w$ belongs to a compact triangle of this decomposition, $d_\mathbb{D}(0,w)\leq d_{\chi,n,b}$ by convexity. The second type of piece consists of half a compact triangle like the previous one, for which the same consideration holds, together with half a 
   Saccheri quatrilateral. If $w$ lies in this last kind of sub-region we have 
$$d_\mathbb{D}\big(w,\gamma(w)\big)=2d_\mathbb{D}(w,g)<2r_{\chi,n,b}$$
    where $g$ is the geodesic containing the corresponding edge of type II and $\gamma$ is the hyperbolic reflection in this edge. As $2r_{\chi,n,b}$ is the first admissible distance by \autoref{admissible_distances}, such a point $w$ cannot be the preimage of an extremal disc center.

    Finally, assume that $w$ belongs to an ideal triangle $T$ with ideal vertex $u=e^{i\theta}$. Then, $T$ is the convex hull of the points $\{0,u,p\}\subset \mathbb{D}$, where $p=\tanh (d_{\chi,n,b}/2)e^{i(\theta\pm\xi)}$ is the remaining vertex of $T$ and $\xi=\beta+\alpha/2$. Notice that $p$ is contained in the closed horoball $B=\{z\in\mathbb{D}:|z-u/2|\leq 1/2\}$ bounded by the horocycle given in \eqref{first_banana_parabolic}, as we can show that $\tanh (d_{\chi,n,b}/2)\leq \cos (\xi)$. Indeed, by applying the hyperbolic law of cosines to $T$ and some classical trigonometric identities, we have the following equalities:
    \begin{align*}\tanh (d_{\chi,n,b}/2)&=\sqrt\frac{\cosh(d_{\chi,n,b})-1}{\cosh(d_{\chi,n,b})+1}=\sqrt\frac{2+\cos\xi-\sqrt 3\sin\xi}{2+\cos\xi+\sqrt 3\sin\xi}=\frac{\sqrt{1+\cos(\xi+\pi/3)}}{\sqrt{1+\cos(\xi-\pi/3)}}\\
    &=\frac{\cos(\xi/2+\pi/6)}{\cos(\xi/2-\pi/6)}\leq \cos(\xi).\end{align*}

    The last inequality holds by the condition $\xi\leq \pi/3$ from \autoref{bound_ideal_triangles} (in fact, the equality is attained if and only if $\xi= \pi/3$). Therefore, $d_\mathbb{D}\big(w,\gamma (w)\big)<2r_{\chi,n,b}$ unless $w=0$ or $w=p$ (and this second case only if  $\xi=\pi/3$), and so generically there cannot be any preimage of an extremal disc center of $S$ in an ideal triangle besides the origin.
\end{proof}

\begin{remark}\label{remark:preimages}
    The proof of the previous lemma exhibits that there cannot be extremal disc centers in some of the pieces of the Dirichlet domain given by an extremal disc configuration. In fact, such centers  can only belong to the shaded pieces of \autoref{fig:preimages}. Moreover, the proper vertex
    of an ideal triangle distinct to the origin can be an extremal disc center if and only if $\tanh (d_{\chi,n,b})=\cos(\xi(r_{\chi,n,b}))$, that is, when $(\chi,n,b)=(-1,3,0)$.

\begin{figure}[!htbp]
    \centering
    \includegraphics[width=0.5\linewidth, draft=False]{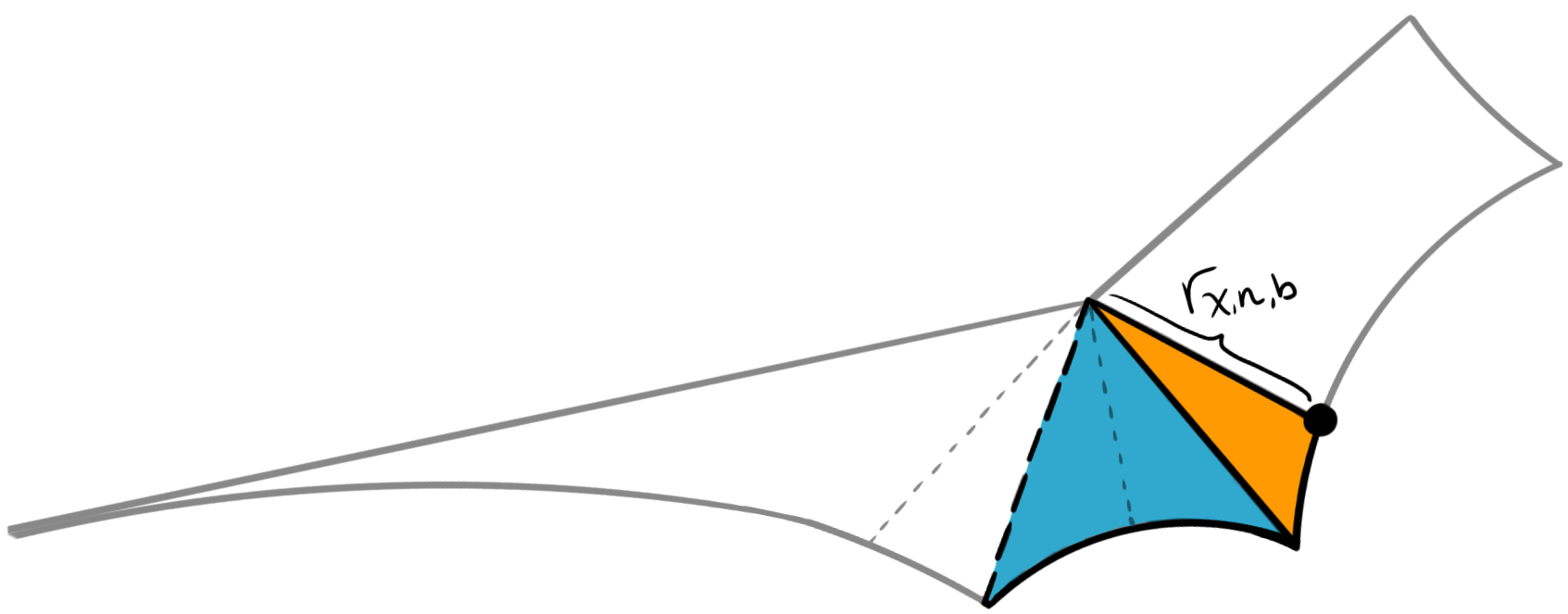}
    \caption{The preimages of the centers of extremal discs are contained in the compact triangular pieces and the filled right-angled subtriangles of the quadrilaterals in the decomposition of \autoref{fig:decomposition}.}
    \label{fig:preimages}
\end{figure}
    
\end{remark}

Another direct consequence of the fact that extremal disc centers stay always away from cusps is the finiteness on the number of extremal discs within a surface:

\begin{proposition} \label{pr:finite}
    The number of extremal discs embedded in a surface is always finite.
\end{proposition}

\begin{proof}
    Let $S\simeq \mathbb{D}/\Gamma$ be an extremal surface with an extremal disc centered at $[0]_\Gamma$. If there is an extremal disc centered at $[w]_\Gamma$, then $w\in D_\Gamma(0)$ must be in the intersection of the finite set of admissible ED-cycles $\{b_{\gamma}(d)\}$, where $\gamma$ is a  side-pairing transformation and 
    $d=d_\mathbb{D}(w,\gamma(w))$ can be assumed to be bounded by 
    $d_0=2(d_{\chi,n,b}+r_{\chi,n,b})$ 
    by \autoref{lemma:distance_origin} and the triangle inequality. Since the admissible ED-cycles associated to distances smaller than $d_0$ give a finite collection of generalized circles, they can only intersect at finitely many points.
\end{proof}

On the process of finding the extremal discs within a given surface, one can often consider a smaller set of admissible ED-cycles than the one in the proof of the previous proposition. This is the case when we can show that the whole domain $D_\Gamma(0)$ is covered by the closures of a given collection of forbidden regions, bounded by certain admissible ED-cycles. Then, the candidates for  extremal disc centers of $S$ are the intersections of the boundary ED-cycles that do not belong to the interior of any of these forbidden regions. This remark has been essential for simplifying the computational work we do in the rest of the paper.

\begin{example}\label{ex:3-punctured_sphere}
Let us consider again the (extremal) 3-punctured sphere $S$: if $p_1=[0]_\Gamma$ is the center of an extremal disc in $S\simeq \mathbb{D}/\Gamma$, the extremal disc configuration  given by $D_\Gamma(0)$ is the one we analyzed in \autoref{ex_sph_three}. Now, \autoref{remark:preimages} ensures that $S$ has at most another extremal disc centered at $p_2=[w]_\Gamma$, where $w\in D_\Gamma(0)$ is any of the three proper vertices. In fact, it can also be seen that these vertices are the only possible candidates by considering the ED-cycles shown in left-hand side of \autoref{fig:trice-punctured-sphere}.
If this second extremal disc does really exist, the extremal disc configuration determined by $D_\Gamma(w)$ is necessarily the same we see in $D_\Gamma(0)$, so there must be an automorphism of $S$ sending $p_1$ to $p_2$. 

\begin{figure}[!htbp]
\begin{subfigure}{.33\textwidth}
  \centering
  \includegraphics[width=0.7\linewidth, draft=False]{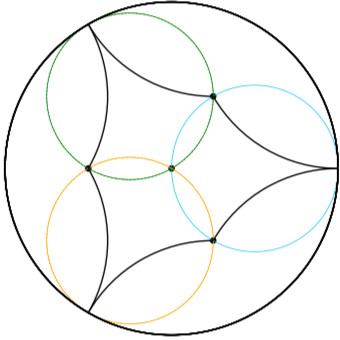}
  \vspace{0.3em}
\end{subfigure}%
\begin{subfigure}{.33\textwidth}
  \centering
  \includegraphics[width=0.73\linewidth, draft=False]{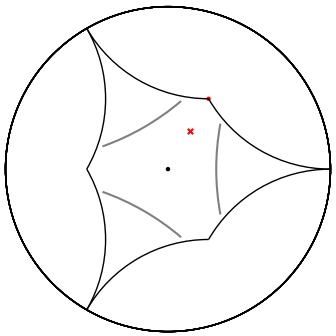}
\end{subfigure}
\begin{subfigure}{.33\textwidth}
  \centering
  \includegraphics[width=0.8\linewidth, draft=False]{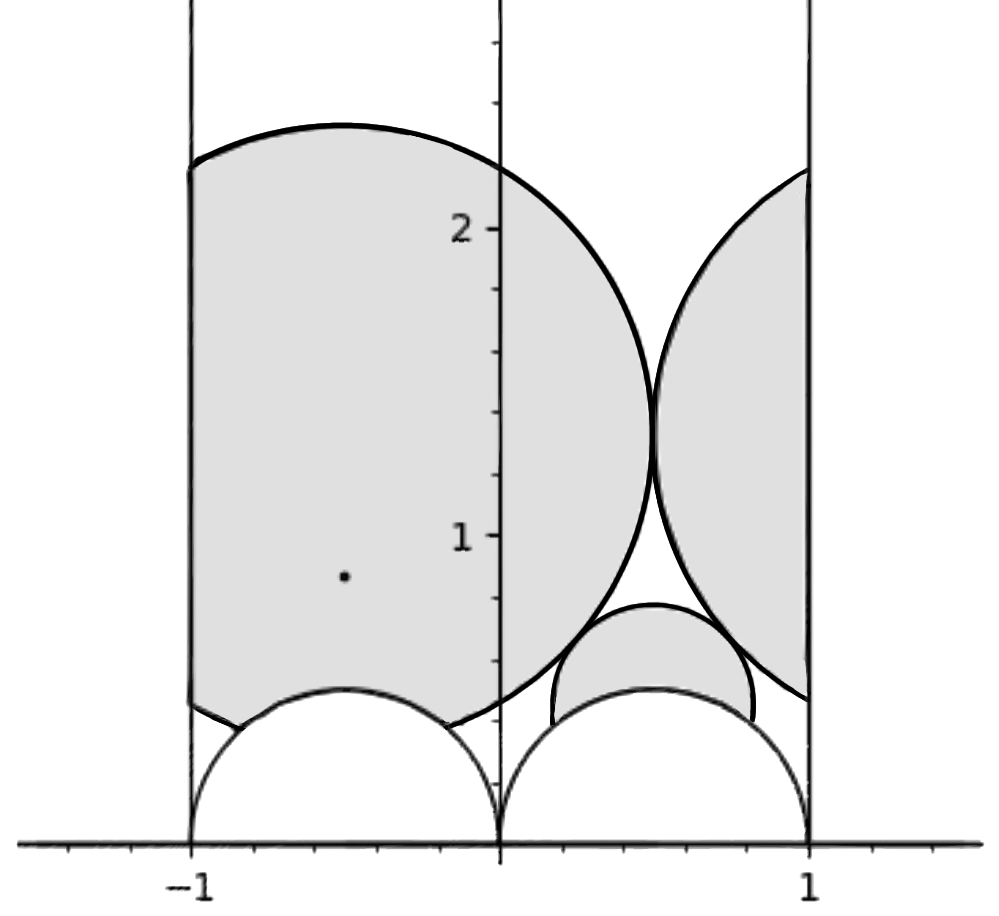}
\end{subfigure}
\caption{At the left, the first admissible ED-cycle for the three distinct (parabolic) side-pairing identifications of $D_\Gamma(0)$. At the middle, the symbol `x' indicates the fixed-point of an elliptic isometry of order $2$ that induces an automorphism that maps the obvious extremal disc to a second hidden one. 
At the right, the shadded region shows an extremal disc in the usual fundamental domain for the equivalent description of   $\widehat{\mathbb{C}}\setminus\{0,1,\infty\}$  as  the quotient space $\mathbb{H}/ \Gamma(2)$. In this model, the two extremal discs lie symmetrically with respect to the imaginary axis, centered at the points $\frac{1}{2} (\sqrt{3} i \pm 1)$.
}
\label{fig:trice-punctured-sphere}
\end{figure}

After the usual normalization, we can assume that $S= \widehat{\mathbb{C}}\setminus\{0,1,\infty\}$. The automorphism group $\text{Aut}(\mathbb{C}\setminus\{0,1,\infty\})$ is the symmetric group of order 6, generated by the order three element $\tau_3(z)= 1/(1-z)$ and the order two element $\tau_2(z)= -1/z$. Clearly, $\tau_3$ corresponds in $S$ to the automorphism $\rho_3$  induced by a rotation of order three around the origin. The fixed points of $\rho_3$, which are precisely $p_1$ and $p_2$, must correspond to the only two values $z\in \widehat{\mathbb{C}}\setminus\{0,1,\infty\}$ such that $z=1/(1-z)$, namely $z=1/2\pm i\sqrt 3/2$. Now, since $\tau_2$ swaps these two values, there exists some automorphism of $S$ swapping $p_1$ and $p_2$. 
In accordance, back to our NEC group description of $S$, one can explicitly show that the elliptic element of order two fixing the point marked `x' in \autoref{fig:trice-punctured-sphere} lies in the normalizer of $\Gamma$. In particular, $p_2$ is in fact the center of a second extremal disc in $S$.
\end{example}

\begin{example}\label{ex:once-punctured-torus}
There is a unique extremal disc configuration with topological type $(+,1,1,0)$ according to the data in \autoref{table:euler-1,-2}, that is, there is a unique extremal once-punctured torus $\mathcal{T}\simeq \mathbb{D}/\Gamma$ with an extremal disc centered at  $p_1=[0]_\Gamma$. The corresponding Dirichlet domain $D_\Gamma(0)$, together with the side-pairing identifications that generate $\Gamma$, are represented at the left-hand side of \autoref{fig:once-punctured-torus}.

One can easily check (see the central picture of \autoref{fig:once-punctured-torus}) that $D_\Gamma (0)$ is covered by the clousure of the forbidden  regions $B_{\gamma_j}(2r_{-1,1,0})$ for $j=1,2,3$ and $B_{\gamma_k}(2r_{-1,1,0},\delta_{-1,1,0})$ for $k=2,3$, where $\gamma_1$ is the parabolic side-pairing  of $D_\Gamma(0)$ and $\gamma_k$ are two of the hyperbolic side-pairings of $D_\Gamma(0)$ for $k=2,3$. We deduce that
the only point $w\in D_\Gamma(0)$ distinct to the origin that can represent an extremal disc center  is the other point marked with a dot in \autoref{fig:once-punctured-torus}. Now, it can be computationally checked that the order $2$ conformal involution that swaps the origin and this point $z$ induces an automorphism of $\mathcal{T}$ that maps $p_1$ to $p_2=[w]_\Gamma \neq p_1$, showing that $\mathcal{T}$ has exactly two extremal discs, centered precisely at these two points.

  \begin{figure}[!htbp]
\begin{subfigure}{.33\textwidth}
  \centering
  \includegraphics[width=0.79\linewidth, draft=False]{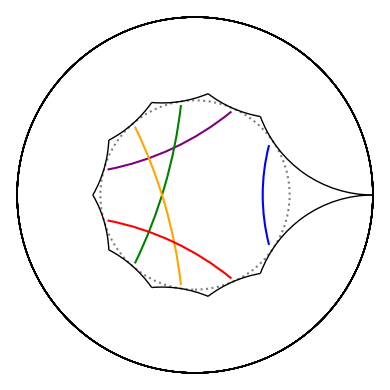}
\end{subfigure}%
\begin{subfigure}{.33\textwidth}
  \centering
  \includegraphics[width=0.75\linewidth, draft=False]{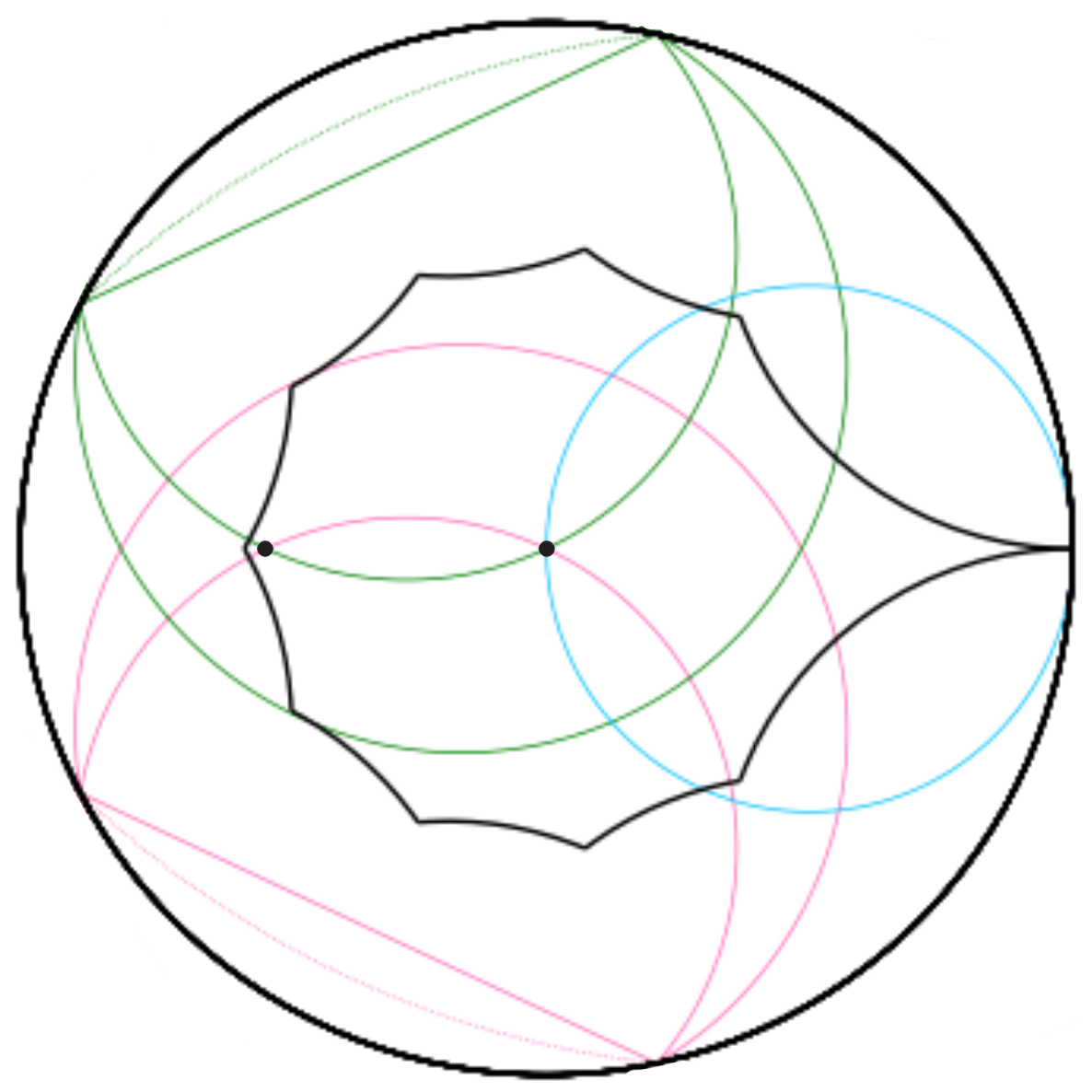}
  \vspace{0.3em}
\end{subfigure}
\begin{subfigure}{.33\textwidth}
  \centering
  \includegraphics[width=0.75\linewidth, draft=False]{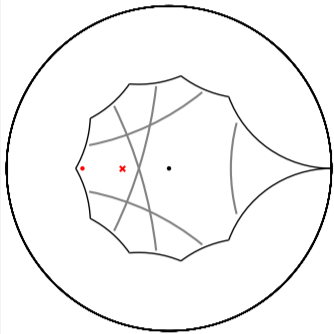}
\end{subfigure}
\caption{At the left, the extremal disc configuration that describes the extremal once-punctured torus $\mathcal{T}$. At the center, the ED-cycles needed to study the extremal disc centers in $\mathcal{T}$. At the right, there are three marked points: the two dots represent the preimages of extremal disc centers in $\mathcal{T}$ and the `x' symbol is the fixed point of an elliptic element of order two that swaps both candidates and induces an automorphism of $\mathcal{T}$.}
\label{fig:once-punctured-torus}
\end{figure}

The understanding of the once-punctured torus $\mathcal{T}$ we have reached by the methods used in this paper has enabled us to obtain a truly remarkable result, which is the computational determination of the complex lattice that corresponds to the compactification of $\mathcal{T}$. That is, we have been able to find, inside the moduli space of genus 1, the location of this torus that is so special for metric reasons. It is the unique complex torus such that the natural (hyperbolic) metric defined in the non-compact surface obtained by removing just one point exhibits an embedded metric disc that is larger than it can be if we do a puncture to any other complex torus. The details will be given in a forthcoming work, where we show that $\mathcal{T}\simeq (\mathbb{C}/\Lambda_\tau)\setminus\{[0]\}$ for  $\Lambda_\tau=\mathbb{Z}+\tau\mathbb{Z}$ with $\tau=-1/2+yi$, $y\approx 1.15005.$
\end{example}

\begin{example} \label{ex:twicepunct_torus}
 
  By computing the ED-cycles in \autoref{fig_oridifferentsurfaces}, it can be shown that there are no hidden extremal disc centers in any of the orientable extremal disc configurations depicted in \autoref{fig_remarkori}. As a consequence, we deduce that their underlying Riemann surfaces cannot be isomorphic. The point here is that an isomorphism of these surfaces would necessarily be induced by an orientation-preserving element $\tau$ that, due to the absence of hidden extremal discs, would send one obvious disc center to the other, and this is obviously not possible.

 \begin{figure}[!htbp]
\begin{subfigure}{.33\textwidth}
  \centering
  \includegraphics[width=.79\linewidth]{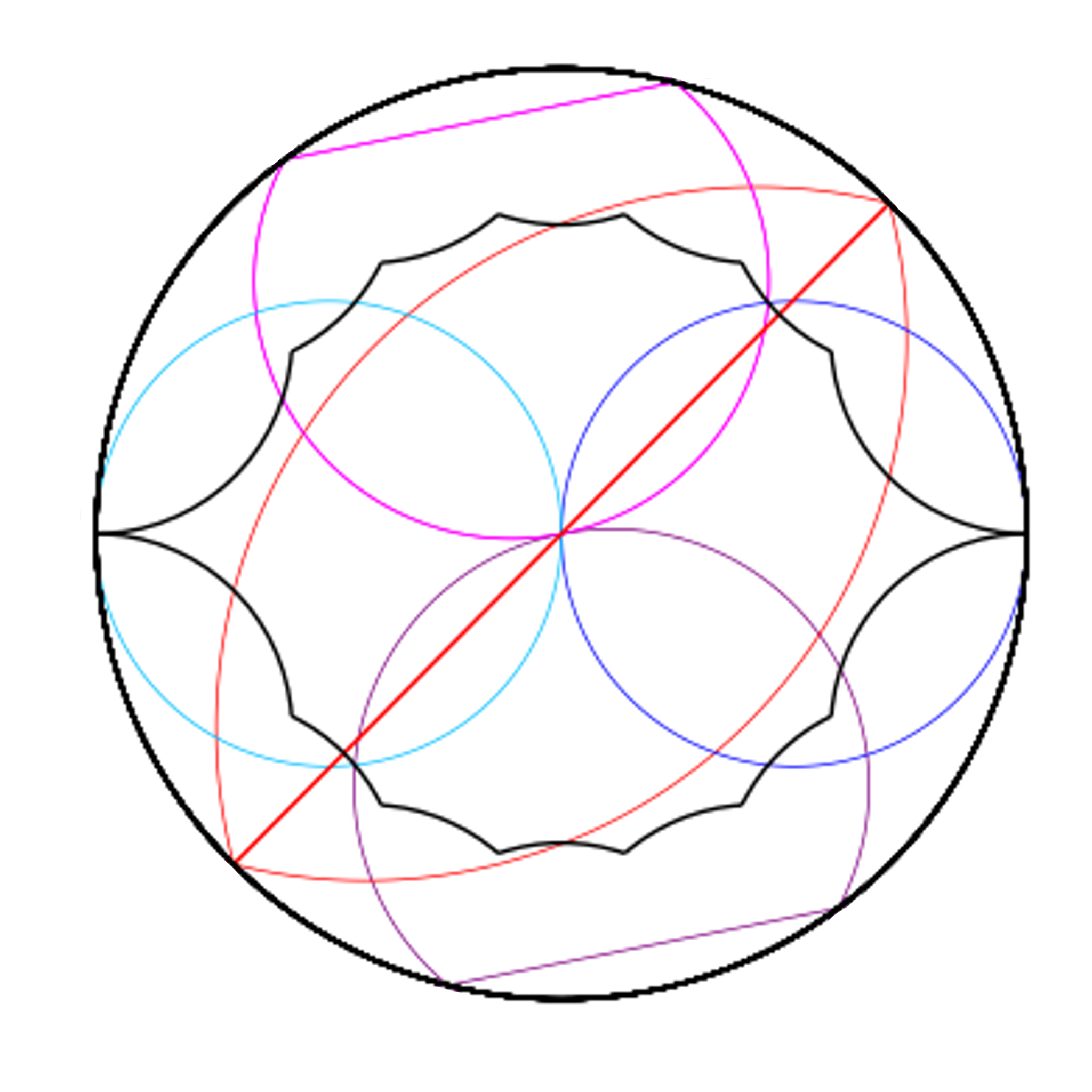}
\end{subfigure}%
\begin{subfigure}{.33\textwidth}
  \centering
  \includegraphics[width=.79\linewidth]{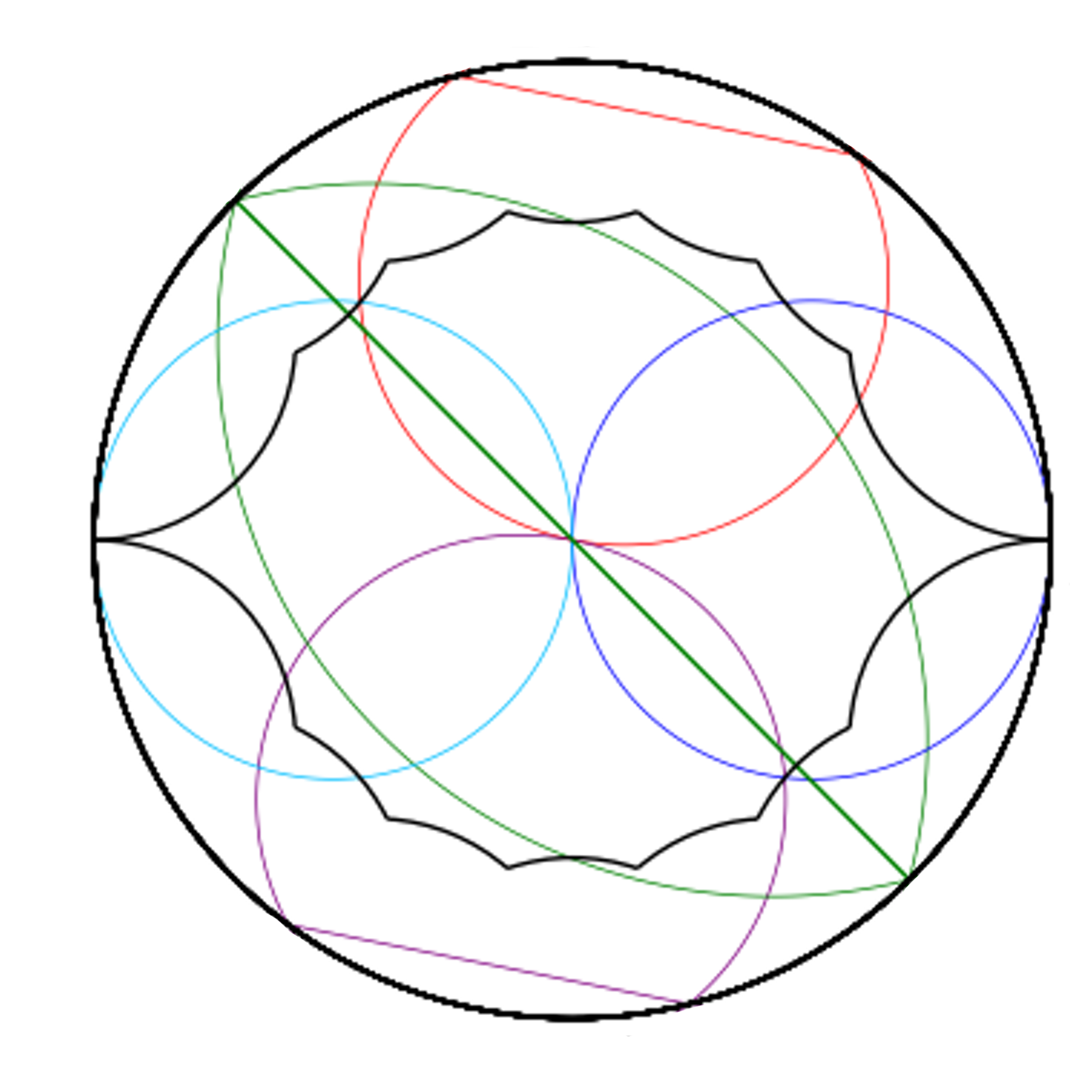}
\end{subfigure}
\caption{Forbidden regions of both orientable $(+,1,2,0)$-configurations in \autoref{fig_remarkori}.}
\label{fig_oridifferentsurfaces}
\end{figure}

\end{example}

\section{A third application: counting extremal discs and extremal surfaces} \label{sec:counting}

The results obtained in Sections \ref{sec:uniformization} to \ref{section:location} provide three complementary ingredients for the study of extremal surfaces. The description of the Dirichlet domains in \autoref{theo2} gives a uniform geometric model for extremal discs, the EBF-correspondence translates extremal disc configurations into a group-theoretic problem, and the theory of admissible ED-cycles supplies effective constraints on the location of hidden extremal discs. The purpose of this section is to combine these ingredients in order to pass from extremal disc configurations to the corresponding extremal surfaces. More precisely, we show how the methods developed in the previous sections can be used to determine all extremal surfaces, count their extremal discs and identify their automorphism groups, and then, we illustrate how this can be done in the case of two specific families of surfaces.

We introduce first the scheme for a general computational method for the location of hidden extremal discs for a given extremal disc configuration. If $S=\mathbb{D}/\Gamma$ is an extremal surface of Euler characteristic $\chi$, $n$ cusps and $b$ boundary components that has an extremal disc centered at $p=[0]_\Gamma$, the procedure to find candidates in $D_\Gamma(0)$ is as follows:

\begin{enumerate}
    \item Compute the complete list of admissible distances $\mathcal{D}_{\chi,n,b}$ for the triplet $(\chi,n,b)$. Taking into account the argument given in the proof of \autoref{pr:finite}, restrict to the subset $\mathcal{D}^*_{\chi,n,b}=\{d\in\mathcal{D}_{\chi,n,b}:d\leq 2(r_{\chi,n,b}+d_{\chi,n,b})\}$.
    \item Consider a subset $\mathfrak{B}$ of the set of admissible ED-cycles  $$\mathfrak{B}_0=\{b_\gamma(d):\gamma \text{ is a side-pairing of }D_\Gamma(0),\, d\in\mathcal{D}^*_{\chi,n,b}\}$$
    such that the closure of the regions bounded by cycles in $\mathfrak{B}$ cover the whole $D_\Gamma(0)$. Denote $\mathfrak{C}\subset D_\Gamma(0)$ the set of points given by the intersections of pairs of cycles in $\mathfrak{B}$ that are not contained in any of the forbidden regions: these are our candidates.
    \item Given $w\in \mathfrak{C}$, if we can find  an isometry $\tilde \psi$ of $\mathbb{D}$ in the normalizer of $\Gamma$ such that the induced automorphism $\psi \in \mathrm{Aut}(S)$ maps $p=[0]_{\Gamma}$ to $q=[w]_{\Gamma}$, $q$ is the center of a hidden extremal disc, and the associated extremal disc configuration coincides with the configuration of the disc centered at $p$.
    \item Even if such an automorphism $\psi$ does not exist, we can still study the Dirichlet domain $D_\Gamma(w)$. If this domain fulfills the conditions of \autoref{theo2} then $[w]_\Gamma$ is the center of an extremal disc in $S$, even if it determines an extremal disc configuration different to the starting one.
\end{enumerate}

We have used SageMath \cite{sagemath} to develop this computational analysis for the extremal disc configurations belonging to the following two families extracted from \autoref{th:tables}:

\begin{itemize}
    \item $\mathcal{F}_1=\{\text{non-closed extremal disc configurations with Euler characteristic } \chi=-1\}$, 
    shown in Figures   \ref{fig:extremalsurfacesg=0_rho=3} and  \ref{fig:extremal_surfaces_1}.
    \item $\mathcal{F}_2=\{\text{extremal disc configurations with Euler characteristic } \chi=-2 \text{ and extended}$ $\text{number of boundary components } \rho=4\}$, displayed in \autoref{fig:extremal_surfaces_3}.
\end{itemize}

The reason behind choosing the family $\mathcal{F}_2$ is that we want to explore how our methods extend the results obtained by Beauchamp  for punctured spheres (\cite{beauchamp2017}), given the fact that, as we proceed to explain now, completing the study for this family does not require too lengthy computations. In any case, extending our classification to all non-closed extremal disc configurations with Euler characteristic $\chi=-2$ will be a natural project to address in the future, but we don't include it here in order to keep the extension of the paper limited, considering the large number of surfaces involved (see \autoref{table:euler-1,-2}). The detailed SageMath codes we have used for our study of the families $\mathcal{F}_1$ and $\mathcal{F}_2$ can be found in \cite{girondo-munozGitHub}. There are some considerations to apply to both families:

Let $\mathcal{C}\in \mathcal{F}_1\cup \mathcal{F}_2$ be an extremal disc configuration realized as the Dirichlet domain $D_\Gamma(0)$, for an  underlying surface $S=\mathbb{D}/\Gamma$ with  Euler characteristic $\chi$, $n$ cusps and $b$ boundary components. If  $d_1=2 r_{\chi,n,b}$ and $d_2$ are the first two admissible distances for $\mathcal{C}$, by \autoref{admissible_distances} these are in fact the first two admissible distances of the whole list  $\mathcal{D}_{\chi,n,b}$. A fortunate thing happens for all configurations in these two families: the closures of some of the forbidden regions $B_\gamma(d_1)$ and $B_\gamma(d_1,d_2)$ always cover the whole domain $D_\Gamma(0)$, so in fact we do not need to compute the list $\mathcal{D}_{\chi,n,b}$. That is, we can study each extremal disc configuration in these families individually, and this fact simplifies the process considerably.

A summary of the results obtained by applying our method to all the configurations in $\mathcal{F}_1$ and $\mathcal{F}_2$ is the following. In all cases, there are at most two candidates that are not $\Gamma$-related. These are the origin and, in some cases, a  second point $w\in D_\Gamma(0)$. These points have been marked with a dot in figures \ref{fig:extremalsurfacesg=0_rho=3} to \ref{fig:extremal_surfaces_3}. In every case in which the second point $w$ exists, we have found that there is an elliptic isometry $\tilde\psi\in N(\Gamma)$ of order $2$ that swaps $w$ and $\eta(0)$ for some $\eta\in\Gamma$, therefore $\tilde\psi$ induces an automorphism $\psi:S\longrightarrow S$. The fixed point of $\tilde\psi$ has been marked with the symbol `x' in figures \ref{fig:extremalsurfacesg=0_rho=3} to  \ref{fig:extremal_surfaces_3} and, if $\eta\neq \text{Id}$, the boundary of $\eta(D_\Gamma(0))$ has been represented with a dotted line.

In other words, whenever we have a second possible disc center $[w]_\Gamma\in S$, there exists an automorphism $\psi$ of the surface $S$ swapping the obvious extremal disc center $[0]_\Gamma$ and $[w]_\Gamma$, so $[w]_\Gamma$ is an extremal disc center
 and the Dirichlet domain $D_\Gamma(w)$ represents the extremal disc configuration $\mathcal{C}$ too. As a consequence, we see that each extremal disc configuration in the families $\mathcal{F}_1$ and $\mathcal{F}_2$ defines a different extremal surface.

 The discussion above shows that the combination of the EBF-correspondence with the ED-cycle method is sufficient to determine all candidates for hidden extremal discs and to decide whether distinct extremal disc configurations arise from the same extremal surface. Consequently, the counting problem for extremal surfaces becomes effective.

Applying this strategy to the families $\mathcal{F}_1$ and $\mathcal{F}_2$ yields a complete
determination of the corresponding extremal surfaces. In particular,
we can decide when different extremal disc configurations correspond
to the same underlying surface and compute the exact number of
extremal discs in each case.

\begin{theorem}\label{th:extremal_surfaces-1}
    There exist exactly 20 non-closed extremal surfaces with Euler characteristic $\chi=-1$, topologically classified as in \autoref{table:euler-1,-2}. Only two of them admit a unique embedded extremal disc, while all the remaining ones have exactly two extremal discs that are related by an automorphism of the surface. 
\end{theorem}

Notice that, in particular, when $\chi=-1$ the number of extremal surfaces agrees with the number of extremal disc configurations. It is for us uncertain whether or not this fact is true in general. Figures \ref{fig:extremalsurfacesg=0_rho=3} and  \ref{fig:extremal_surfaces_1} describe these  surfaces. 

\begin{figure}[!htbp]
\begin{subfigure}{.24\textwidth}
  \centering
\includegraphics[width=0.99\linewidth]{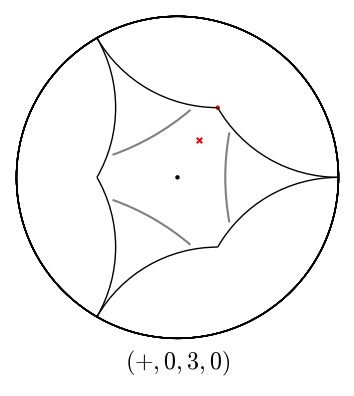}
\end{subfigure}
\begin{subfigure}{.24\textwidth}
\includegraphics[width=0.99\linewidth]{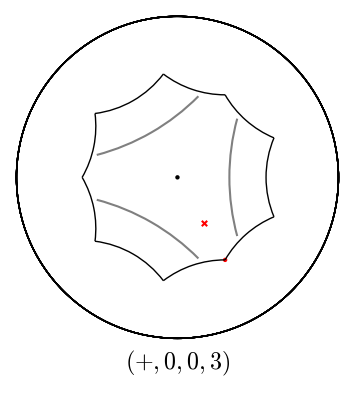}
\end{subfigure}
\begin{subfigure}{.24\textwidth}
\includegraphics[width=0.99\linewidth]{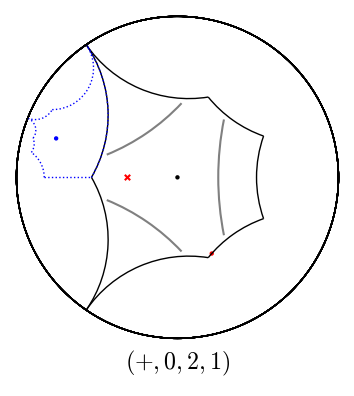}
\end{subfigure}
\begin{subfigure}{.24\textwidth}
\includegraphics[width=0.99\linewidth]{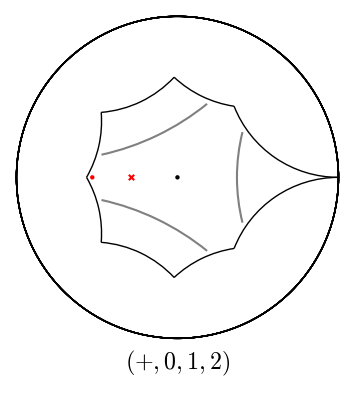}
\end{subfigure}%
\caption{Case $\chi=-1$: the four extremal surfaces of genus 0 with $\rho=3$ and their extremal disc centers.}
\label{fig:extremalsurfacesg=0_rho=3}
\end{figure}

\begin{figure}[!htbp]
\begin{subfigure}{.24\textwidth}
  \centering
  \includegraphics[width=\linewidth]{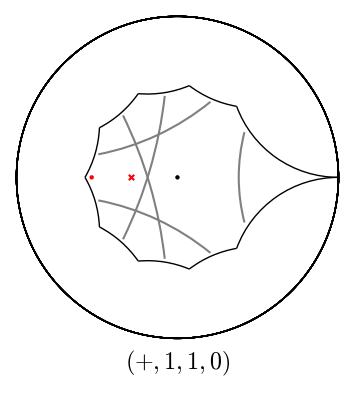}
\end{subfigure}
\begin{subfigure}{.24\textwidth}
  \centering
  \includegraphics[width=\linewidth]{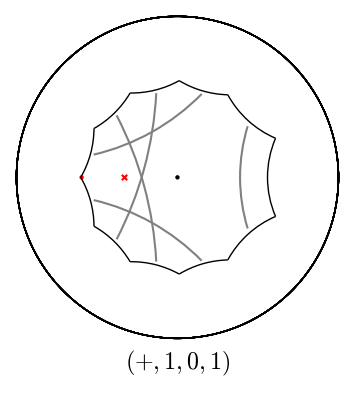}
\end{subfigure}
\begin{subfigure}{.24\textwidth}
  \centering
  \includegraphics[width=\linewidth]{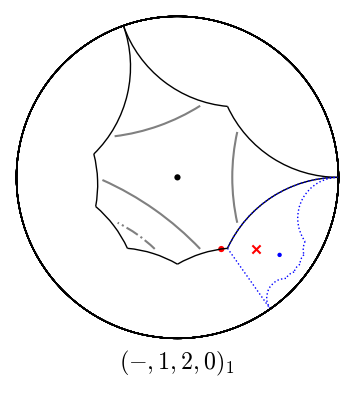}
\end{subfigure}
\begin{subfigure}{.24\textwidth}
  \centering
  \includegraphics[width=\linewidth]{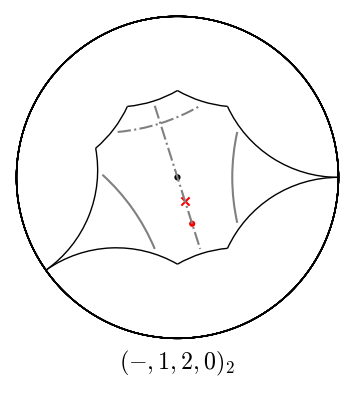}
\end{subfigure}

\begin{subfigure}{.24\textwidth}
  \centering
  \includegraphics[width=\linewidth]{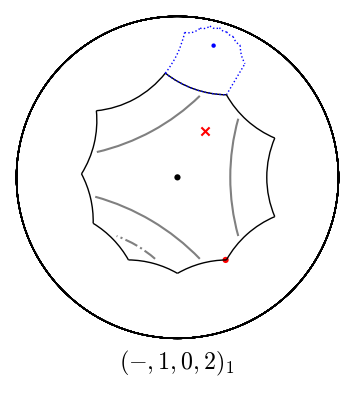}
\end{subfigure}
\begin{subfigure}{.24\textwidth}
  \centering
  \includegraphics[width=\linewidth]{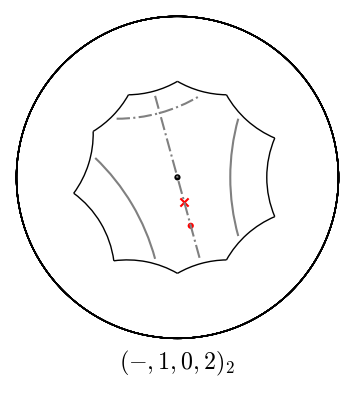}
\end{subfigure}
\begin{subfigure}{.24\textwidth}
  \centering
  \includegraphics[width=\linewidth]{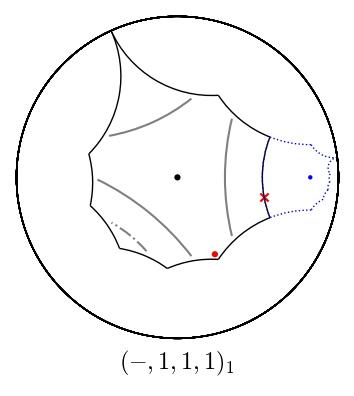}
\end{subfigure}
\begin{subfigure}{.24\textwidth}
  \centering
  \includegraphics[width=\linewidth]{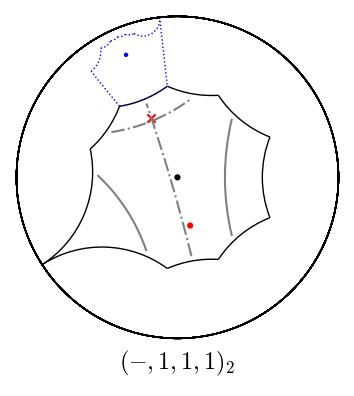}
\end{subfigure}

\begin{subfigure}{.24\textwidth}
  \centering
  \includegraphics[width=\linewidth]{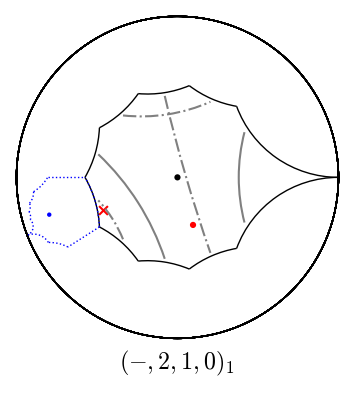}
\end{subfigure}
\begin{subfigure}{.24\textwidth}
  \centering
 \includegraphics[width=\linewidth]{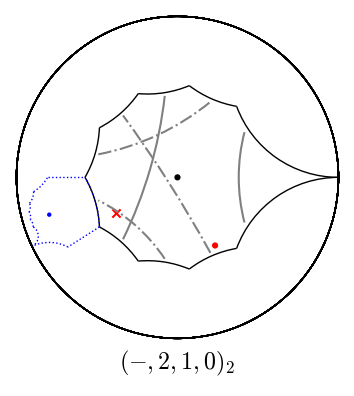}
\end{subfigure}
\begin{subfigure}{.24\textwidth}
  \centering
  \includegraphics[width=\linewidth]{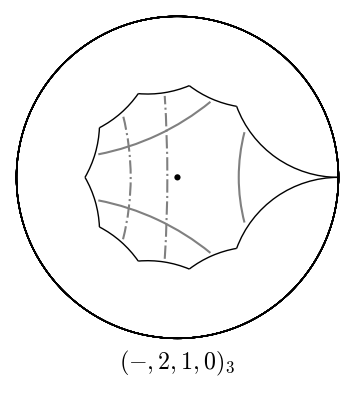}
\end{subfigure}
\begin{subfigure}{.24\textwidth}
  \centering
  \includegraphics[width=\linewidth]{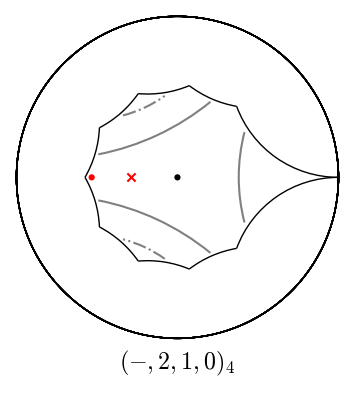}
\end{subfigure}

\begin{subfigure}{.24\textwidth}
  \centering
  \includegraphics[width=\linewidth]{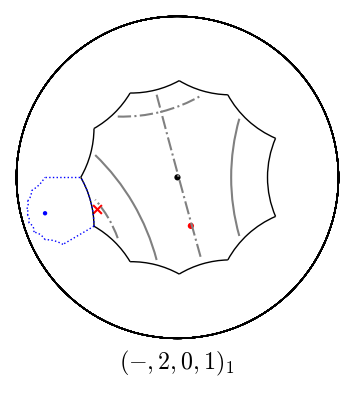}
\end{subfigure}
\begin{subfigure}{.24\textwidth}
  \centering
  \includegraphics[width=\linewidth]{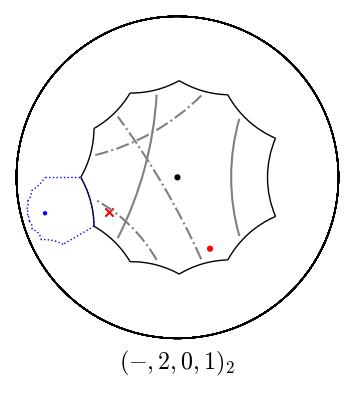}
\end{subfigure}
\begin{subfigure}{.24\textwidth}
  \centering
  \includegraphics[width=\linewidth]{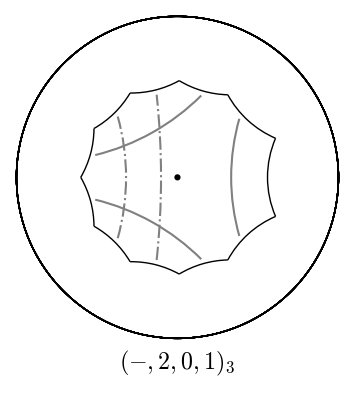}
\end{subfigure}
\begin{subfigure}{.24\textwidth}
  \centering
  \includegraphics[width=\linewidth]{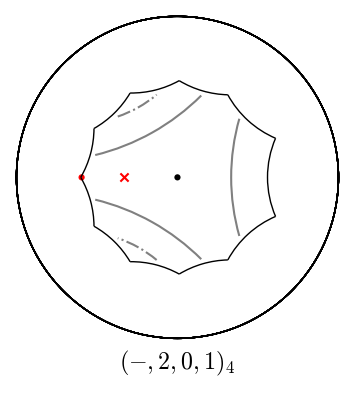}
\end{subfigure}
\caption{All the remaining extremal surfaces with  Euler characteristic $\chi=-1$, and their extremal disc centers.}
\label{fig:extremal_surfaces_1}
\end{figure}

\begin{theorem}\label{th:extremal_surfaces-2}
    There exist exactly 7 different extremal spheres with $n$ cusps and $b$ boundary components such that $n+b=4$, topologically classified as in \autoref{table:euler-1,-2}. If $n\equiv 0\mod 4$, the surface has two extremal discs, while if $n\equiv 1 \mod 2$, the surface has a single extremal disc. In the case $n=2$ and $b=2$, two of the three surfaces have two extremal discs while the remaining one has a single extremal disc. The surfaces with two extremal discs admit always an automorphism that swaps them.
\end{theorem}
 \autoref{fig:extremal_surfaces_3} describe all these extremal spheres. This theorem can be seen as the natural extension of the statement in [\cite{beauchamp2017}, Theorem 4.9].

\begin{figure}[!htbp]
    \begin{subfigure}{.24\textwidth}
  \centering
  \includegraphics[width=\linewidth]{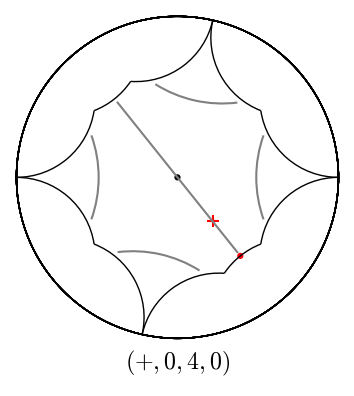}
\end{subfigure}%
\begin{subfigure}{.24\textwidth}
  \centering
  \includegraphics[width=\linewidth]{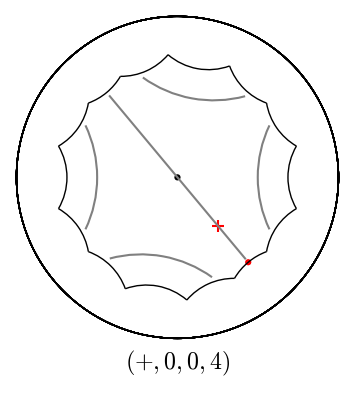}
\end{subfigure}
\begin{subfigure}{.24\textwidth}
  \centering
  \includegraphics[width=\linewidth]{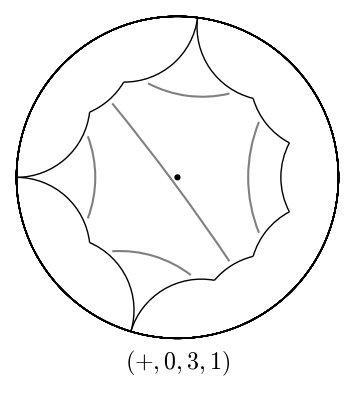}
\end{subfigure}%
\begin{subfigure}{.24\textwidth}
  \centering
  \includegraphics[width=\linewidth]{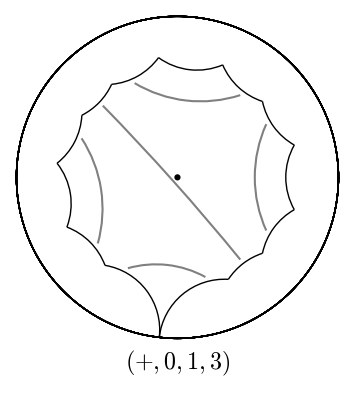}
\end{subfigure}

\begin{subfigure}{.24\textwidth}
  \centering
  \includegraphics[width=\linewidth]{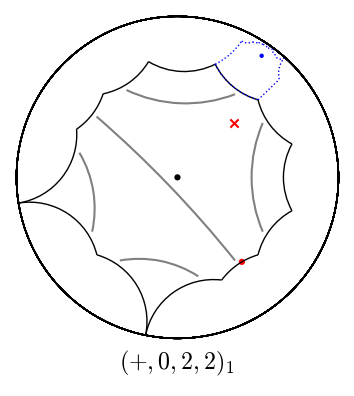}
\end{subfigure}
\begin{subfigure}{.24\textwidth}
  \centering
  \includegraphics[width=\linewidth]{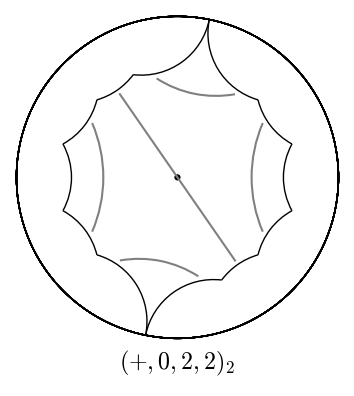}
\end{subfigure}
\begin{subfigure}{.24\textwidth}
  \centering
  \includegraphics[width=\linewidth]{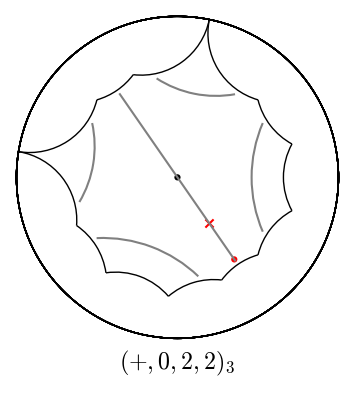}
\end{subfigure}
\caption{Extremal surfaces and extremal disc centers for the topological type $(+,0,n,b)$, with extended number of boundary components $n+b=4$.}
\label{fig:extremal_surfaces_3}
\end{figure}

\subsection{Identification of the automorphism groups}

The explicit location of the extremal disc centers in the  families $\mathcal{F}_1$ and $\mathcal{F}_2$ helps us to identify the group of automorphisms of each of these surfaces, since automorphisms preserve the set of extremal disc centers.
Recall that, according to \autoref{th:extremal_surfaces-1} and \autoref{th:extremal_surfaces-2}, 
 if $S\simeq \mathbb{D}/\Gamma$ is such an extremal surface with an extremal disc centered at $p_1=[0]_\Gamma$, $S$ can have at most another extremal disc centered at $p_2=[z_2]_\Gamma$. Therefore, we can distinguish two classes of automorphisms of $S$:

\begin{itemize}
    \item  If such a point $p_2 \in S$ exists and   $\psi$ is the automorphism of $S$ swapping $p_1$ and $p_2$, whose existence we already know, we can have another $\varphi \in \text{Aut}(S)$ that  verifies $\varphi (p_1)=p_2$ but then $\varphi\circ\psi\in \text{Aut}(S)$ fixes $p_1$.
    \item In any other case, any $\varphi \in \text{Aut}(S)$ fixes $p_1$, and $\varphi$ has a lift $\tilde \varphi$ such that $\tilde \varphi (0)=0$ and $\tilde \varphi(D_\Gamma (0))=D_\Gamma(0)$.
\end{itemize}

Consequently, the determination of the whole automorphism group is based on checking what transformations $\tilde \varphi$ fixing the origin, which can be either elliptic isometries or hyperbolic reflections, lie in the normalizer of $\Gamma$.  Observe that, in accordance to \autoref{rem:orient_withoutboundary}, for orientable surfaces without boundary we restrict to orientation preserving automorphisms and, in particular,  $\tilde\varphi $ can only be an elliptic isometry.

In the cases where a hidden extremal disc center $p_2\in S$ exists, the point $\tilde \varphi (z_2)\in D_\Gamma(0)$ must project into $p_2$ too, and we see that:

\begin{itemize}
    \item If $z_2$ belongs to the interior of $D_\Gamma(0)$, $\tilde \varphi$ can only be the hyperbolic reflection in the geodesic connecting the origin and $z_2$.
    \item Otherwise, $\tilde \varphi$ can be an elliptic isometry of orders $3$ or $2$ (depending on whether $z_2$ is a vertex or not) or any reflection preserving the $\Gamma$-orbit of $z_2$ in $D_\Gamma(0)$.
\end{itemize}

Once we have determined which isometries could induce automorphisms, we can decide if any of them actually do by checking if they normalize the corresponding NEC (or Fuchsian) group uniformizing the surface. 
Doing this for each of the extremal surfaces in figures \ref{fig:extremalsurfacesg=0_rho=3},   \ref{fig:extremal_surfaces_1} and \ref{fig:extremal_surfaces_3}, one can show that the automorphism groups are as shown in \autoref{th:auts}.

\begin{table}[!htbp]
    \centering
    \begin{tabular}{|c|c|c|c|}
    \hline
    Surface $S$ & $\mbox{Aut(S)}$ & Surface $S$ & \mbox{Aut(S)} \\
        \hline
    $ (+,0,3,0)$  &  $S_3$ & $(+,0,0,3)$ & $D_6$    \\    \hline    
            $(+,0,2,1)$    & $\mathcal{C}_2 \times C_{\overline{2}}$   & $(+,0,1,2)$  &  $\mathcal{C}_2 \times C_{\overline{2}}$  \\     \hline
            $(+,1,1,0) $   & $\mathcal{C}_2$  & $(+,1,0,1)$  &  $\mathcal{C}_2 \times C_{\overline{2}}$ \\
                \hline
            $(-,1,2,0)_1$    & $\mathcal{C}_2 \times C_{\overline{2}}$  & $(-,1,2,0)_2$  & $\mathcal{C}_2 \times C_{\overline{2}}$ \\
                \hline
            $(-,1,0,2)_1$ & $\mathcal{C}_2 \times C_{\overline{2}}$     & $(-,1,0,2)_2$   & $\mathcal{C}_2 \times C_{\overline{2}}$ \\
                \hline
                $(-,1,1,1)_1$ & $\mathcal{C}_2 $     & $(-,1,1,1)_2$   & $\mathcal{C}_2$  \\
                \hline
            $(-,2,1,0)_1$   &  $\mathcal{C}_2 $ & $(-,2,1,0)_2 $  & $\mathcal{C}_2 $ \\
                \hline
            $(-,2,1,0)_3$ & $ C_{\overline{2}}$     & $(-,2,1,0)_4 $  & $\mathcal{C}_2 \times C_{\overline{2}}$ \\
                \hline
                $(-,2,0,1)_1$   &  $\mathcal{C}_2 $ & $(-,2,0,1)_2 $  & $\mathcal{C}_2 $ \\
                \hline
            $(-,2,0,1)_3$ & $ C_{\overline{2}}$     & $(-,2,0,1)_4 $  & $\mathcal{C}_2 \times C_{\overline{2}}$ \\
                \hline
            $(+,0,4,0)$    & $\mathcal{C}_2 \times C_{2}$   & $(+,0,0,4)$   & $\mathcal{C}_2 \times C_{\overline{2}}
            \times C_{\overline{2}}$  \\
                \hline
$(+,0,3,1)$   & 1 & $(+,0,1,3)$      &  1 \\
\hline 
 $(+,0,2,2)_1$    &  $\mathcal{C}_2 \times C_{\overline{2}}$  & $(+,0,2,2)_2$   & $C_2$ \\
                \hline
  $(+,0,2,2)_3$    &  $\mathcal{C}_2 \times C_{\overline{2}}$  & \   & \ \\
      \hline    
    \end{tabular}
    \caption{Automorphism groups of the extremal surfaces in $\mathcal{F}_1$ and $\mathcal{F}_2$.
    }
    \label{th:auts}
\end{table}

 The case of $(+,0,3,0)$ has been explained in detail in \autoref{ex:3-punctured_sphere}. Now, for $(+,0,0,3)$ we readly see that the automorphisms induced by elements fixing the origin form an index 2 subgroup isomorphic to $D_3$, and the full group is obtained adding the order 2 automorphism $\psi$ switching both extremal disc centers. As there are order 6 automorphisms in this group and at least two order $2$ automorphisms, it must be isomorphic to the dihedral group $D_6$.
 
 For the rest of the cases in \autoref{th:auts}, $\mathcal{C}_2$ stands for a cyclic group generated by  the order two element swapping both extremal disc centers, and the factors $C_2$ and $C_{\overline{2}}$ in the table denote cyclic groups of order two generated, respectively, by the rotation fixing the origin and by an obvious mirror symmetry fixing a line through the origin.

\

 {\bf Acknowledgements}.-Work partially supported by grant CEX2023-001347-S, funded by MCIN/AEI/10.13039/501100011033, Spain, and by the FPU Graduate Research Grant FPU24/01163.

\bibliographystyle{alpha}
\bibliography{references}  

\end{document}